\documentclass[11pt]{article}
\usepackage{amsmath, amssymb, theorem, latexsym, epsfig}
\usepackage{eufrak}
\numberwithin{equation}{section}

\theoremstyle{plain}
\theorembodyfont{\itshape}
\newtheorem{theorem}{Theorem}[section]
\newtheorem{proposition}[theorem]{Proposition}
\newtheorem{lemma}[theorem]{Lemma}
\newtheorem{corollary}[theorem]{Corollary}
\newtheorem{convention}[theorem]{Convention}
                                                                                
\theorembodyfont{\rmfamily}
\newtheorem{definition}[theorem]{Definition}

\newtheorem{example}[theorem]{Example}
\newtheorem{remark}[theorem]{Remark}
\def\rr{\mathbb R}
\newenvironment{proof}{{\noindent \textbf{Proof}\,\,}}{\hspace*{\fill}$\Box$\medskip}
\title{Density of  thin film  billiard reflection pseudogroup in Hamiltonian 
symplectomorphism pseudogroup}
\author{Alexey Glutsyuk\thanks{
CNRS, France (UMR 5669 (UMPA, ENS de Lyon), UMI 2615 (ISC J.-V.Poncelet)). E-mail: 
aglutsyu@ens-lyon.fr} \thanks{HSE University, Moscow, Russia}\thanks{Kharkevich Institute for Information Transmisson Problems (IITP RAS), Moscow, 
Russia} \thanks{The author is partially supported by Laboratory of Dynamical Systems and Applications, HSE University, of the Ministry of science and higher education of the RF grant ag. No 075-15-2019-1931} \thanks{Partially supported by RFBR and JSPS (research project 19-51-50005)}}
\begin{document}
\maketitle
\begin{abstract} Reflections from hypersurfaces act by symplectomorphisms on the space of oriented lines with respect to the 
canonical symplectic form. 
We consider an arbitrary $C^{\infty}$-smooth 
hypersurface $\gamma\subset\rr^{n+1}$ that is either a global strictly  convex 
closed hypersurface, or a germ of hypersurface. 
We deal with the pseudogroup generated by compositional ratios 
of reflections from $\gamma$ and of reflections from 
its small deformations. In the case, when $\gamma$ is a global convex hypersurface, 
we show that the $C^{\infty}$-closure of 
the latter pseudogroup contains the 
pseudogroup of Hamiltonian diffeomorphisms between domains in the phase cylinder: 
 the space of oriented lines intersecting $\gamma$ transversally. We prove an analogous local result 
 in the case, when $\gamma$ is a germ. The derivatives of the  above compositional differences in the deformation 
 parameter are Hamiltonian vector fields calculated by Ron Perline. To prove the main results, we 
 find the Lie algebra generated by them and prove its $C^{\infty}$-density in the Lie algebra of Hamiltonian vector fields. 
 We also prove analogues of the above results for hypersurfaces in Riemannian manifolds.
\end{abstract}
\tableofcontents
\def\cc{\mathbb C}
\def\wh#1{\widehat#1}
\def\oc{\overline{\cc}}
\def\cp{\mathbb{CP}}
\def\wt#1{\widetilde#1}
\def\rr{\mathbb R}
\def\var{\varepsilon}
\def\tt{\mathcal T}
\def\var{\varepsilon}
\def\mce{\mathcal E}
\def\mcf{\mathcal F}
\def\mcfnab{\mcf_{ab}|_{N_{ab}}}
 \def\mcg{\mathcal G}
 \def\so{\operatorname{SO}}
 \def\dist{\operatorname{dist}}
\def\ha{\hat a}
\def\hb{\hat b}
\def\hc{\hat c}
\def\hd{\hat d}
\def\ind{\operatorname{ind}}
\def\nn{\mathbb N}
\def\mcd{\mathcal D}
\def\mcp{\mathcal P}
\def\rp{\mathbb{RP}}
\def\mct{\mathcal T}
\def\zz{\mathbb Z}
\def\re{\operatorname{Re}}
\def\im{\operatorname{Im}}
\def\La{\Lambda}
\def\gg{\EuFrak g}
\def\gh{\EuFrak H}
\def\ggg{\EuFrak G_{glob}}
\def\ggo{\EuFrak G_{glob,0}}
\def\mod{\operatorname{mod}}
\def\RR{\rr}
\def\sym{\operatorname{Sym}}
\def\skt{S^k(T^*\gamma)}
\def\ttg{T^*\gamma}
\def\la{\lambda}
\def\mcq{\mathcal Q}
\def\supp{\operatorname{supp}}
\def\ker{\operatorname{Ker}}
\def\lll{\mathbb L}
\def\oo{\operatorname{O}}
\def\so{\operatorname{SO}}
\def\gl{\operatorname{GL}}

\section{Main results: density of thin film planar billiard reflection pseudogroup}

It is well-known that billiard reflections acting on the space of oriented geodesics 
preserve the canonical symplectic form \cite{ar2, ar3, mm, melrose1, melrose2, tab95}.  
However only a tiny part of symplectomorphisms 
are realized by reflections. There  is an important open question stated in \cite{ptt}: 
which symplectomorphisms can be realized by compositions of reflections?

In the present paper we consider the case of  billiards in Euclidean spaces. Namely, we deal with 
a hypersurface $\gamma\subset\rr^{n+1}$: either a strictly convex closed hypersurface, or a germ of 
hypersurface. We investigate compositional ratios 
of reflections from $\gamma$ and reflections from its small deformations. They were 
 introduced and studied by Ron Perline \cite{perline}. He had shown that 
  their derivatives in the parameter are Hamiltonian vector fields and calculated their 
Hamiltonian functions.  See Subsection 1.4 below.

We show that the Lie algebra generated by the above Hamiltonian vector fields is dense in the 
Lie algebra of all the Hamiltonian vector fields (symplectic vector fields in the case, when 
$\gamma$ is a germ of curve). 
We apply this result to the pseudogroup generated by the above compositional ratios of reflections. 
 In the case of a strictly convex closed hypersurface we show that 
 the $C^{\infty}$-closure of 
the latter pseudogroup contains the pseudogroup of Hamiltonian 
symplectomorphisms between open subdomains of the phase cylinder: 
the space of oriented lines intersecting $\gamma$ transversally.  We prove analogous statement for germs 
of hypersurfaces. The corresponding results are stated in Subsections 1.1 and 1.2 respectively.

In the case of a germ of planar curve, when $n=1$, we show that  
the above pseudogroup coming from reflections is $C^{\infty}$-dense in the  pseudogroup 
of  symplectomorphisms between simply connected subdomains of a small region in the 
space of oriented lines.  See the statements of the corresponding results in Subsection 1.3. 

The above results on $C^{\infty}$-closures of pseudogroups are 
 proved in Section 4.

For the proof of main results we find the Lie algebra generated by the above-mentioned 
Hamiltonian functions from \cite{perline} and prove its $C^{\infty}$-density in the 
 space of $C^{\infty}$-smooth functions. The corresponding results are presented in  Subsection 1.4 
 and proved in Sections 2 and 3 in the cases, when $n=1$ and $n\geq2$ respectively.
 
In Subsection 1.5 we state analogues of Perline's formula for Hamiltonian 
function and of main results for hypersurfaces in Riemannian manifolds. 
We prove them in Section 5.
 
 In Subsection 1.6 we present  a brief historical survey 
 and state an open problem.

\subsection{Main density results: case of global hypersurface}
Here we deal with billiards in  $\rr^{n+1}$, $n\geq1$: reflections from hypersurfaces, which act 
on the space of oriented lines in $\rr^{n+1}$, see the next paragraph. 
 It is well-known that the space of all  oriented lines is canonically 
diffeomorphic to the tangent bundle of the unit sphere $S^{n}$, and it carries a canonical symplectic form 
induced from the standard symplectic form on $TS^{n}$.  All the reflections from  hypersurfaces 
act symplectomorphically \cite{ar2, ar3, mm, melrose1, melrose2, tab95}. 

Let  $\gamma\subset\rr^{n+1}$ be a closed strictly convex hypersurface. 
Recall that the reflection $\mct_{\gamma}$ from $\gamma$ acts on the space of oriented lines in 
$\rr^{n+1}$ as follows. If a line $L$ is either disjoint from 
$\gamma$, or tangent to it, then it is fixed by $\mct_{\gamma}$. If $L$ intersects $\gamma$  transversally, 
then we take its last point $A$ of intersection with $\gamma$ (in the sense of orientation of the line $L$). We  
set $\mct_{\gamma}(L)$ to be its  image under reflection from the tangent hyperplane $T_A\gamma$ and orient 
it by a vector in $T_A\rr^{n+1}$ directed inside the compact domain bounded by $\gamma$. 

Let now $\vec N$ be the exterior unit normal vector field on the hypersurface $\gamma$. Let  
$f:\gamma\to\rr$ be a $C^{\infty}$-smooth function. Consider the following 
deformation of the  surface $\gamma$: 
\begin{equation}\gamma_\var=\gamma_{\var,f}:=\{ x+\var f(x)\vec N(x) \ | \ x\in\gamma\}.\label{gaef}\end{equation}
We deal with reflections $\mct_{\gamma}$ and $\mct_{\gamma_\var}=\mct_{\gamma_{\var,f}}$ acting on oriented lines. Recall that the {\it phase cylinder} is the domain 
$$\Pi=\Pi_\gamma:= \ \{\text{the oriented lines intersecting } \ \gamma \text{ transversally}\},$$
\begin{equation}\Delta\mct_\var=\Delta\mct_{\var,f}:=\mct_{\gamma_\var}^{-1}\circ\mct_\gamma, \ \ \ 
v_f:=\frac{d\Delta\mct_\var}{d\var}|_{\var=0}.\label{detevf}\end{equation}
For every compact subset $K\Subset\Pi$ the symplectic mapping $\Delta\mct_\var$ is well-defined on $K$, whenever 
$\var$ is small enough dependently on $K$.  
Hence, $v_f$ are symplectic vector fields on $\Pi$. They are Hamiltonian with 
Hamiltonian functions given in \cite[p.623]{perline}, see also (\ref{hfw}). We prove the following theorem.

\begin{theorem} \label{tlidense} The Lie algebra  generated by the vector fields $v_f$, $f\in C^{\infty}(\gamma)$, 
is $C^{\infty}$-dense in the Lie algebra of Hamiltonian vector fields on $\Pi_\gamma$: dense in the topology of uniform convergence with all derivatives  on compact subsets.. 
\end{theorem}

Using Theorem \ref{tlidense}, we prove  density results for pseudogroups generated by the compositional 
differences $\Delta\mct_\var$. To state them, let us recall the following well-known definitions. 

 Consider a collection of $C^{\infty}$-smooth 
 diffeomorphisms  between domains on some 
 manifold (e.g., the space of oriented lines in $\rr^{n+1}$). We consider each diffeomorphism 
 together with its domain of definition and all its restrictions to smaller domains. 
 We will deal with those  compositions of diffeomorphisms and of their inverses 
  that are well-defined on some domains. Recall that the collection of 
 all the above compositions is called a {\it pseudogroup}. (In the case when we take 
 only compositions of mappings and not of their inverses, the above collection of compositions 
 is called a {\it pseudo-semigroup.}) 
 
 \begin{definition} For a given 
 sequence of domains $V_n$ we say  that the {\it intersections} $V_n\cap V$ 
 {\it converge} to $V$, if each compact subset $K\Subset V$ is contained in $V_n$ whenever $n$ 
 is large enough (depending on $K$).  Let $g$ be a $C^{\infty}$-smooth diffeomorphism 
 defined on a domain $V$: we deal with it as with the pair $(g,V)$. Recall that a sequence $g_n$ of elements in a pseudogroup  
 converges to a mapping $g$ on $V$ 
 in the $C^{\infty}$-topology, if $g$ is well-defined on $V$, 
 $g_n$ are defined on a sequence of domains $V_n$ 
 such that $V_n\cap V\to V$, and $g_n\to g$ uniformly on compact subsets in $V$ with 
 all the derivatives. The {\it $C^{\infty}$-closure} of a given pseudogroup consists of the mappings 
 forming the pseudogroup and the limits of the above converging sequences. 
 (The latter closure is a pseudogroup itself.)
 \end{definition}

For every positive-valued mapping $\delta:C^{\infty}(\gamma)\to\rr_+$  set 
\begin{equation}\mcg(\delta):=\text{ the pseudogroup generated by the collection of mappings} \label{pseudod}\end{equation}
$$\{\Delta\tt_{\var,h} \ | \ h\in C^{\infty}(\gamma), 0\leq\var\leq\delta(h)\}.$$
Let us recall the following well-known definition. 
\begin{definition} \label{defmham} Let $M$ be a symplectic manifold, and let $V\subset M$ be  
an open domain. A symplectomorphism $F:V\to V_1\subset M$ is ($M$-) {\it Hamiltonian,} if there exists a smooth family 
$F_t:V\to V_t\subset M$ of symplectomorphisms parametrized by $t\in[0,1]$, $V_0=V$,  
$F_0=Id$, $F_1=F$,
 such that for every $t\in[0,1]$ the derivative $\frac{dF_t}{dt}$ is a Hamiltonian vector field on $V_t$. 
 (In the case when $V=V_t=M$, this definition coincides with the 
 usual definition of Hamiltonian diffeomorphism of a manifold  \cite[definition 4.2.4]{ban}.)
 \end{definition}
 \begin{remark} Let  $M$ be a two-dimensional 
 topological cylinder equipped with a symplectic form, and let $V\subset M$ be a 
subcylinder with compact closure and smooth boundary. Further assume that  $V$ is a deformation retract of the 
ambient cylinder $M$. Then not every area-preserving map $F:V\to U\subset M$ 
 is Hamiltonian. A necessary condition 
for being Hamiltonian is that for every boundary component $L$ of the subcylinder $V$ 
the self-intersecting domain bounded by $L$ and  by its image $F(L)$ has zero signed 
area. This follows  from results presented in \cite[chapters 3, 4]{ban}. 
 The above-defined $M$-Hamiltonian symplectomorphisms between domains in a symplectic manifold  form a pseudogroup: 
 a composition of two $M$-Hamiltonian symplectomorphisms $U\to V$ and $V\to W$ is $M$-Hamiltonian. 
\end{remark}

\begin{theorem} \label{tc1} For every  mapping 
$\delta:C^{\infty}(\gamma)\to\rr_+$ the $C^{\infty}$-closure of the pseudogroup $\mcg(\delta)$ 
contains the whole pseudogroup of $\Pi$-Hamiltonian diffeomorphisms 
between domains in $\Pi$. In other words, for every 
domain $V\subset\Pi$ and every $\Pi$-Hamiltonian symplectomorphism $F:V\to W\subset\Pi$ there exists 
a sequence $F_n$ of finite compositions of mappings $\Delta\mct_{\var_j,h_j}$, $0\leq\var_j\leq\delta(h_j)$, $F_n$ being 
defined on domains $V_n$  with $V_n\cap V\to V$, such that $F_n\to F$ on $V$ in the 
$C^{\infty}$-topology. 
\end{theorem}

\begin{definition} \label{akclo} We say that two hypersurfaces are $(\alpha,k)$-close if the distance 
between them in the $C^k$-topology is no greater than $\alpha$. This means that 
the hypersurfaces are diffeomorphically parametrized by the same manifold, and their 
parametrizations can be chosen  so that the $C^k$-distance between them is no greater than $\alpha$. 
\end{definition}

\begin{corollary} \label{c1} For every $\alpha>0$ (arbitrarily small) and $k\in\nn$ 
 the $C^{\infty}$-closure of the pseudogroup generated by reflections from hypersurfaces $(\alpha,k)$-close to 
$\gamma$ contains the whole pseudogroup of $\Pi$-Hamiltonian symplectomorphisms  
between domains in $\Pi$. 
\end{corollary}

Theorem \ref{tlidense} will be proved in Subsection 2.5 (for $n=1$) and 3.3 (for $n\geq2$). 
Theorem \ref{tc1} and Corollary \ref{c1}  
will be proved in Section 4.

\subsection{Main results in the local case}

Here we consider the case, when $\gamma$ is a germ of (not necessarily convex) 
$C^{\infty}$-smooth hypersurface in $\rr^{n+1}$ 
at some point $x_0$  and state local versions of the above results. 

Each function $f:\gamma\to\rr$ defines a deformation $\gamma_\var=\gamma_{\var,f}$ of the 
germ $\gamma$ given by (\ref{gaef}). Let $\Delta\mct_\var$, $v_f$ be the same, as in (\ref{detevf}). 
Fix an arbitrary small contractible neighborhood of the point $x_0$ in $\gamma$, and now denote by $\gamma$ the 
latter neighborhood. It is a local hypersurface parametrized by a contractible  domain in $\rr^n$. 
 Fix an arbitrary domain $\Pi=\Pi_\gamma$ in the space of oriented lines in $\rr^{n+1}$ such that each point in $\Pi$  
represents a line $L$ intersecting $\gamma$ transversally at one point $x=x(L)$. Let $w=w(L)\in T_x\gamma$ denote 
the orthogonal projection to $T_x\gamma$ of the directing vector of $L$ at $x$. We identify all the tangent 
spaces $T_x\gamma$ by projecting them orthogonally to an appropriate coordinate subspace  
$\rr^n\subset\rr^{n+1}$. 
For simplicity in the local case under consideration we will choose  $\Pi$ to be 
diffeomorphic to a contractible domain in $\rr^{2n}$ via the correspondence sending $L$ to $(x,w)$.

\begin{theorem} \label{tlidense2} The Lie algebra of vector fields on $\Pi_\gamma$ generated by the fields $v_f$, $f\in C^{\infty}(\gamma)$, see (\ref{detevf}), is $C^{\infty}$-dense in the Lie algebra of Hamiltonian vector fields on 
$\Pi_\gamma$. 
\end{theorem}

The definitions of pseudogroup $\mcg(\delta)$ and $(\alpha,k)$-close hypersurfaces for local hypersurfaces (germs) 
are the same, as in the previous subsection. 

\begin{theorem} \label{tc2}  The statements of Theorem \ref{tc1} and Corollary \ref{c1} hold in the 
case, when $\gamma$ is a local hypersurface: a $C^{\infty}$-smooth germ of hypersurface in $\rr^{n+1}$.
 \end{theorem}
 
 Theorems \ref{tlidense2} will be proved in Section 2.5 (for $n=1$) and 3.3 (for $n\geq2$). 
Theorem \ref{tc2} will be proved in Section 4.

\subsection{Case of  germs of planar curves: density in symplectic vector fields and maps}
Let now $\gamma$ be a $C^{\infty}$-smooth germ of planar curve in $\rr^2$ at a point $O$. Let $\Pi$ be 
the same, as in the previous subsection. 

\begin{theorem} \label{tlidense3} In the case, when $\gamma$ is a germ of planar curve,  
 the Lie algebra generated by the fields $v_f$ is dense in the Lie algebra of symplectic vector fields on $\Pi$. 
\end{theorem}
\begin{theorem} \label{tds3} If $\gamma$ is a germ of planar curve, then 
 the statements of Theorems \ref{tc1} and Corollary \ref{c1} hold with density in the pseudogroup of 
symplectomorphisms between simply connected domains in $\Pi$.
\end{theorem}
Theorem \ref{tlidense3} is proved in Subsection 2.5. Theorem \ref{tds3} is proved in Subsection 4.2.

\subsection{Hamiltonian functions for vector fields $v_f$ and  the Lie algebra generated by them}
 Let us  recall a well-known presentation of the symplectic structure 
on $\Pi$ coming from the identification of the outwards directed unit tangent bundle on 
$\gamma$ with the unit ball bundle $T_{<1}\gamma\subset T\gamma$ \cite{ar2, ar3, mm, melrose1, melrose2, tab95}.
To each oriented line $\la$ intersecting $\gamma$ transversally we put into correspondence 
the pair $(x(\la),u(\la))$, where $x(\la)$ is its last intersection point  with $\gamma$  
in the sense of orientation of the line $\la$, and $u=u(\la)\in T_{x(\la)}\rr^{n+1}$ is the unit vector directing $\la$. 
To a unit vector $u\in T_x\rr^{n+1}$, $x\in\gamma$, we put into correspondence its orthogonal 
projection $w:=\pi_\perp(u)\in T_x\gamma$ to the tangent hyperplane of $\gamma$ at $x$; one has 
$||w||\leq1$. In the case, when $\gamma$ is a strictly convex closed hypersurface, the composition of the 
above correspondences yields a diffeomorphism
\begin{equation}J:\Pi\mapsto T_{<1}\gamma, \ \  \la\mapsto(x(\la),w(\la))\in T_{<1}\gamma:=\{(x,w)\in T\gamma \ | \ ||w||<1\}.\label{idgeo}\end{equation}
In the case, when $\gamma$ is a germ, $J$ is a diffeomorphism of the corresponding domain $\Pi$ in the 
space of oriented lines onto an open subset $J(\Pi)\subset T_{<1}\gamma$. 
The tangent bundle $T\gamma$ consists of pairs $(x,w)$, $x\in\gamma$, $w\in T_x\gamma$, 
and carries the Liouvillian 1-form $\alpha\in T^*(T\gamma)$ defined as follows: for every 
$x\in\gamma$, $w\in T_x\gamma$ and $v\in T_{(x,w)}(T\gamma)$ one has 
\begin{equation}\alpha(v)=<w,\pi_*(v)>, \ \ \ \pi_*=d\pi, \ \ \pi \text{ is the  projection }  T\gamma\to\gamma.
\label{liof}\end{equation}
 The {\it standard symplectic form} on $T\gamma$ is given by the differential 
 $$\omega:=d\alpha.$$
 The above Liouville form $\alpha$ and symplectic form $\omega$ are well-defined on every Riemannian 
 manifold $\gamma$. 
 
 The above diffeomorphism $J$ is known to be a symplectomorphism  \cite{ar2, ar3, mm, melrose1, melrose2, tab95}.
 In what follows we switch from $\Pi$ to $T_{<1}\gamma$; the images of the vector fields $v_f$ under the  
 symplectomorphism $J$ are symplectic vector fields on $T_{<1}\gamma$, which will be also denoted by $v_f$. 
\begin{remark} Consider the  correspondence $\la\mapsto(x(\la),u(\la))$  from the beginning of the subsection 
 between $\Pi$ and a domain in the restriction to $\gamma$ of the unit tangent bundle of the ambient space 
 $\rr^{n+1}$. 
For every function $f$ on $\gamma$ the  vector field $v_f$ on $\Pi$ given by (\ref{detevf}) 
is identified via the latter correspondence with a well-defined vector field on the space of all pairs $(x,u)$, where $x\in\gamma$ and $u\in T_x\rr^{n+1}$ 
is a unit vector transversal to $\gamma$. (See the next theorem and remark.) Therefore, the $J$-pushforward of the field $v_f$  is a well-defined vector field 
on all of $T_{<1}\gamma$, which will be also denoted by $v_f$. 
{\it In what follows, whenever the contrary is not specified, we deal with $v_f$ as with the latter vector field on $T_{<1}\gamma$.} 
 \end{remark}
 \begin{theorem} \label{forh} \cite[p.623]{perline} For every $C^{\infty}$-smooth function $f:\gamma\to\rr$ the corresponding vector field 
 $v_f$ is Hamiltonian with the Hamiltonian function 
 \begin{equation}H_f(x,w):=-2\sqrt{1-||w||^2}f(x).\label{hfw}\end{equation}
 \end{theorem}
 \begin{remark} The formula from \cite[p.623]{perline} was given in the chart $(x,u)$, with the 
Hamiltonian function  $H_f(x,u)=-2<u,\vec N(x)>f(x)$. The latter scalar product being equal to $\sqrt{1-||w||^2}$, $w=\pi_\perp(u)$,  this yields (\ref{hfw}).
 \end{remark} 
 \begin{proposition} \label{pcom} 
 The vector space over $\rr$ generated by the Poisson brackets of the functions $H_f(x,w)$ 
 given by (\ref{hfw}) for all $f\in C^{\infty}(\gamma)$ is a Lie algebra (under Poisson bracket), where each element can be represented as a sum of at most $2n+1$ 
 Poisson brackets. It consists  of the functions on $T_{<1}\gamma$ of type $\eta(w)$, 
 where $\eta$ is an arbitrary $C^{\infty}$-smooth 1-form on $\gamma$, and is identified with the space of 1-forms. 
 The Lie algebra structure thus obtained on the 1-forms is isomorphic  to  the  
 Lie algebra of all $C^{\infty}$-smooth vector fields (with Lie bracket) via the duality isomorphism $T^*\gamma\to T\gamma$ 
 given by the metric.
 \end{proposition}
 Proposition \ref{pcom} is proved in Subsection 3.1.
 
 \begin{theorem} \label{tlidf} Let $\gamma$ be an arbitrary Riemannian manifold (neither necessarily embedded, 
 nor necessarily compact). Let $\gh$ denote the Lie algebra generated (under the Poisson bracket) by the
  Hamiltonian functions (\ref{hfw}) on $T_{<1}\gamma$ constructed from all  $f\in C^{\infty}(\gamma)$. The Lie algebra $\gh$ 
 is $C^{\infty}$-dense in the Lie algebra of all the $C^{\infty}$-functions on 
 the unit ball bundle $T_{<1}\gamma$: dense in the topology of uniform  convergence with all derivatives on compact subsets. 
\end{theorem}

 Below we formulate more precise versions of Theorem \ref{tlidf}, which give 
 explicitly the Lie algebra generated by functions (\ref{hfw}) in different cases. 
To obtain its simpler description,  we deal with the following renormalization isomorphism
 \begin{equation} Y:T_{<1}\gamma\to T\gamma, \ (x,w)\mapsto(x,y),\  
 y=y(w):=\frac w{\sqrt{1-||w||^2}}\in T_x\gamma.\label{yw}
 \end{equation}
 This equips $T\gamma$ with the pushforward symplectic structure $Y_*\omega$. 
 \begin{convention}\label{conve}
 Sometimes we rewrite functions in $(x,w)\in T_{<1}\gamma$ as functions of $(x,y)\in T\gamma$. 
 By definition, the Poisson bracket of two functions in $(x,y)$ is calculated with respect to the pushforward 
 symplectic form $Y_*\omega$. It coincides with the pushforward by $Y$ of their Poisson bracket 
 as functions of $(x,w)$, with respect to the canonical symplectic form $\omega$ on $T_{<1}\gamma$. 
 \end{convention} 
 \begin{example}  For every function $f$ on $\gamma$ the corresponding Hamiltonian function $H_f(x,w)$ 
 of the vector field $v_f$ written in the new coordinate $y$ is equal to 
 \begin{equation} H_f(x,w)=-2H_{0,f}(x,y), \ \ \ H_{0,f}(x,y):=\frac{f(x)}{\sqrt{1+||y||^2}}.\label{hfxy}\end{equation}
 This follows from (\ref{hfw}) and (\ref{yw}). 
 \end{example}
  Let $S^k(T^*\gamma)$ denote the space of those $C^{\infty}$-smooth functions on $T\gamma$ whose 
 restrictions to the fibers $T_x\gamma$ are homogeneous polynomials of degree $k$. 
 For every $\phi\in S^k(T^*\gamma)$ we write $\phi=\phi(y)$  as a polynomial in $y\in T_x\gamma$ 
 with coefficients depending $C^{\infty}$-smoothly on $x\in\gamma$. Set 
 \begin{equation}H_{k,\phi}:=\frac{\phi(y)}{\sqrt{1+||y||^2}}, \ \ H_{k,\phi}\in C^{\infty}(T\gamma),\label{hkhf}\end{equation}
 $$\La_k:=\{ H_{k,\phi} \ | \ \phi\in S^k(T^*\gamma)\}.$$
 \begin{example} \label{excurves} Let $n=1$ and let, for simplicity, $\gamma$ be connected. Then $\gamma$ is either 
 an interval, or a circle, equipped with a Riemannian metric. Let $s$ be the length element on $\gamma$. 
 Then $y$ is just one variable, each element in $S^k(\ttg)$ is a product $\phi=h(s)y^k$, $h\in C^{\infty}(\gamma)$. 
 In this case, when $n=1$, we will use  simplified notations replacing $\phi$ by $h$ and writing  $H_{k,\phi}$ as 
 $H_{k,h}$: 
 \begin{equation}H_{k,\phi}=H_{k,h}:=\frac{h(s)y^k}{\sqrt{1+y^2}}; \ \ \ \ \La_k=\{ H_{k,h} \ | \ h\in C^{\infty}(\gamma)\}.
 \label{hdh}\end{equation} 
 In the case, when $\gamma$ is a closed planar curve with induced metric, the phase cylinder $\Pi$ is indeed 
 a cylinder: it is diffeomorphic to the product of a circle and an interval via the following correspondence. To 
 each oriented line $L$ intersecting $\gamma$ transversally we put into correspondence its last   
 intersection point with $\gamma$ (identified with its natural parameter $s$), running all of $\gamma\simeq S^1$,
   and the intersection angle $\theta\in(0,\pi)$. 
  The above-defined symplectic form on $\Pi$ is equal to $\sin\theta ds\wedge d\theta$ (see \cite[lemma 3.7]{tab}). The 
  natural parameter yields a canonical trivialization of the tangent bundle to $\gamma$. After this trivialization 
  the corresponding vectors $w$ and $y$ become just real numbers that are equal to 
  \begin{equation} w=\cos\theta, \ \ \ \ \ y=\cot\theta.\label{tant}\end{equation}
 \end{example}
 
 In the renormalized coordinates $y$ the Lie algebra $\gh$ admits the following  description. 
 
 \medskip
 
\begin{theorem} \label{tlidf1} Let $\gamma$ be the same, as in Theorem \ref{tlidf}. 
If $n=\dim\gamma\geq2$, then one has
\begin{equation}\gh=\oplus_{k\geq0}\La_k.\label{oplk}\end{equation}
\end{theorem}

\begin{theorem} \label{tlidf2} Statement (\ref{oplk}) remains valid in the case, when $n=1$ and $\gamma$ is an interval 
equipped with a Riemannian metric.
\end{theorem}

It appears, that in the case, when $n=1$ and $\gamma$ is a topological circle, statement (\ref{oplk}) is not true. 
To state its version in this special case, let us introduce the 
following notations. Let $d\in\zz_{\geq0}$, $H_{d,h}$ be functions on the phase cylinder $\Pi$ given 
by formula (\ref{hdh}), and 
let $\La_d$ be their space, see (\ref{hdh}). Let us consider that the length of the curve $\gamma$ is equal to 
$2\pi$, rescaling the metric by constant factor.  Set 
 $$\La_{d,0}:= \{ H_{d,h}\in\La_d \ | \ \int_0^{2\pi}h(s)ds=0\}.$$
 For every odd polynomial vanishing at zero with derivative, 
 \begin{equation} P(y)=\sum_{j=1}^ka_{j}y^{2j+1},\label{podd}\end{equation}
set  
\begin{equation}\wt P(x):=x^{-\frac12}P(x^{\frac12})=\sum_{j=1}^ka_{j}x^j.\label{wtp}
\end{equation}

\begin{theorem} \label{tlidf3}  Let $n=1$ and $\gamma$ be a topological 
 circle equipped with a Riemannian metric. 
  The Lie algebra  $\gh$ generated by the functions $H_{0,f}$, $f\in C^{\infty}(\gamma)$, see (\ref{hfxy}) and (\ref{hdh}), is 
 \begin{equation}\ggg:=\La_1\oplus(\oplus_{d\in 2\zz_{\geq0}}\La_d)\oplus(\oplus_{d\in2\zz_{\geq1}+1}\La_{d,0})\oplus\Psi,\label{algg}\end{equation}
\begin{equation}\Psi:=\{\frac{P(y)}{\sqrt{1+y^2}} \ | \ P(y) \text{ is a polynomial as in 
(\ref{podd}) with } \wt P'(-1)=0\}.\label{dpsi}\end{equation}
\end{theorem}
Theorems \ref{tlidf}, \ref{tlidf1}, \ref{tlidf2}, \ref{tlidf3} will be proved in Sections 2 (for $n=1)$ and 3 (for  $n\geq2$).  

\subsection{Case of hypersurfaces in  Riemannian manifolds}

Let $M$ be a complete Riemannian manifold. Let $\gamma\subset M$ be a closed strictly convex hypersurface bounding 
a domain $\Omega\subset M$ homeomorphic to a ball. Let for every geodesic $\Gamma$ lying in a small neighborhood 
$U=U(\overline\Omega)\subset M$ the intersection $\Gamma\cap\Omega$ be either empty, or an interval bounded 
by two points of transversal intersections with $\partial\Omega$. The space of geodesics intersecting $\Omega$ will be called the 
{\it phase cylinder} and denoted by $\Pi$. The billiard ball map $\mct_\gamma$ of reflection from $\gamma$ 
 acts on the space  of oriented geodesics in the same way, as in Subsection 1.1. This action is symplectic 
 with respect to the standard symplectic form on the space of oriented geodesics that is given by 
 symplectic reduction (Melrose construction, see \cite{ar2, ar3, mm, melrose1, melrose2, tab95}). 
 The phase cylinder $\Pi$ is a symplectic manifold symplectomorphic to the unit ball bundle $T_{<1}\gamma$ equipped with the standard symplectic form, as in Subsection 1.4. 
 
 Let us repeat  Perline's thin film billiard construction. 
 For every $C^{\infty}$-smooth function $f:\gamma\to\rr$ consider the   family of 
hypersurfaces $\gamma_\var$ consisting of the points $\gamma_\var(x)$ 
defined as follows. For every $x\in\gamma$  consider the geodesic $\Gamma_N(x)$ through 
the point $x$ that is orthogonal to $\gamma$ and directed out of $\Omega$.  The point $\gamma_\var(x)$ 
is obtained from the point $x$ by shift of (signed) distance $\var f(x)$ along the geodesic $\Gamma_N(x)$; 
one has $\gamma_\var(x)\notin\overline\Omega$, if $f(x)>0$, and $\gamma_\var(x)\in\Omega$, if $f(x)<0$. 

\begin{theorem} \label{forh1} Consider the following compositional ratio and its derivative
$$\Delta\mct_\var=\Delta\mct_{\var,f}:=\mct_{\gamma_\var}^{-1}\circ\mct_{\gamma}, \ \ 
v_f:=\frac{d\Delta\mct_{\var,f}}{d\var}|_{\var=0}.$$
The derivative $v_f$  is a Hamiltonian vector field on the phase cylinder $\Pi\simeq T_{<1}\gamma$ with the 
same Hamiltonian function $H_f(x,w)=-2\sqrt{1-||w||^2}f(x)$, as in (\ref{hfw}).
\end{theorem}
Theeorem \ref{forh1} follows easily  from its Euclidean version due to R.Perline (Theorem \ref{forh} in 
Subsection 1.4). 

We will deal not only with the global case, when $\gamma$ is a closed hypersurface, as above, 
but also with the case, when $\gamma$ is a germ of hypersurface. 
\begin{theorem} \label{theriem} The statements of Theorem \ref{tc1} and Corollary \ref{c1} hold for any hypersurface $\gamma$ 
as above and for any germ of hypersurface in any Riemannian manifold.
\end{theorem}
Proofs of Theorems \ref{forh1} and \ref{theriem} will be given in Section 5.
\subsection{Historical remarks and an open problem}
In 1999 R.Peirone  studied dynamics of billiard in thin film formed by a hypersurface $\gamma$ and its 
given deformation $\gamma_\var$ with small $\var$. He proved the following transitivity result for small 
$\var$: {\it for every two points $p_1,p_2\in \gamma$ there exists 
a billiard orbit starting at $p_1$ that lands at $p_2$ after sufficiently many reflections} \cite{peirone}. 
A series of results on  dynamics in  thin film billiard, including results 
mentioned in Subsection 1.1 (calculation of the vector field $v_f$ and its Hamiltonian function), 
together with relation to geodesic flow 
were obtained in \cite{perline}. For other results and open 
problems, e.g., on relations to integrable PDE's, see \cite[sections 8, 9]{perline} and references to \cite{perline}.

 Corollary \ref{c1} states that the $C^{\infty}$-closure of the 
 {\it pseudogroup} generated by reflections from hypersurfaces close to $\gamma$ 
  contains the whole pseudogroup of $\Pi$-Hamiltonian diffeomorphisms between domains 
 in the phase cylinder $\Pi=\Pi_\gamma$. That is, each $\Pi$-Hamiltonian 
 diffeomorphism is the limit of a sequence of compositions of {\it reflections and 
 their inverses.} 
 \medskip
 
 {\bf Open Problem.} Is it true that for every closed strictly convex hypersurface $\gamma\subset\rr^{n+1}$ the 
 $C^{\infty}$-closure of the {\it pseudo-semigroup} generated by reflections from the hypersurface $\gamma$ and 
 from its small deformations {\it (without including their inverses)} contains the whole pseudogroup of $\Pi$-Hamiltonian diffeomorphisms between domains in the phase 
 cylinder $\Pi$?
 
\section{The Lie algebra in the case or curves. Proof of Theorems \ref{tlidf2}, \ref{tlidf3}, \ref{tlidf}, \ref{tlidense3}, \ref{tlidense}, \ref{tlidense2}}

In the present section we consider the case, when $\gamma$ is a connected curve equipped with a 
Riemannian metric. We prove  
Theorems \ref{tlidf2}, \ref{tlidf3} and Theorem \ref{tlidf} for curves. To do this, first in Subsection 2.1 we calculate Poisson brackets of  functions of type $H_{d,h}$ from (\ref{hdh}). 
Then we treat separately two cases, when $\gamma$ is respectively either an interval (Subsection 2.2), or a circle 
(Subsections 2.3, 2.4), and prove Theorems \ref{tlidf2}, \ref{tlidf3}, 
 \ref{tlidf}. Then we deduce Theorems \ref{tlidense}, \ref{tlidense2}, \ref{tlidense3} in Subsection 2.5. 

\subsection{Calculation of Poisson brackets}
We work in the space $T_{<1}\gamma=\gamma\times(-1,1)$ equipped with coordinates $(s,w)$. Here $s$ is the 
natural parameter of the curve $\gamma$, and $w\in(-1,1)$ is the coordinate of tangent vectors to $\gamma$ 
with respect to the basic vector $\frac{\partial}{\partial s}$. We identify a point of the curve $\gamma$ with the 
corresponding parameter $s$. Recall that the canonical 
symplectic structure of $T_{<1}\gamma$ is the standard symplectic structure $dw\wedge ds$. 
Therefore the Poisson bracket of two functions $F$ and $G$ is equal to 
\begin{equation}\{ F, G\}=\frac{\partial F}{\partial w}\frac{\partial G}{\partial s}-\frac{\partial F}{\partial s}\frac{\partial G}{\partial w}.\label{poisbr1}\end{equation}
We write formulas for Poisson brackets of  functions from (\ref{hdh}) in the coordinates 
$$(s,y), \ \ \ y=\frac w{\sqrt{1-w^2}},$$
in which they take simpler forms.  Recall that for every $d\in\zz_{\geq0}$ and every function $h(s)$ we set  
$$H_{d,h}(s,y):=\frac{y^d}{\sqrt{1+y^2}}h(s),  \ \  \ H_{-1,h}(s,y):=0,$$ 
see (\ref{hdh}), and for every function $h(s)$ on $\gamma$ the vector field $v_h$ on $T_{<1}\gamma$ 
 is Hamiltonian with the Hamiltonian function $-2H_{0,h}=H_{0,-2h}$, see (\ref{hfxy}). 

\begin{proposition} For every $d,k\in\zz_{\geq0}$ and 
any two functions $f(s)$, $g(s)$ one has 
\begin{equation}\{H_{d,f},H_{k,g}\}=H_{d+k-1, dfg'-kf'g}+
H_{d+k+1, (d-1)fg'-(k-1)f'g}.\label{poissform}
\end{equation}
\end{proposition}
\begin{proof} For every $m\in\zz_{\geq0}$ one has 
$$\frac{y^m}{\sqrt{1+y^2}}=\frac{w^m}{(\sqrt{1-w^2})^{m-1}},$$
$$\frac{\partial}{\partial w}\left(\frac{y^m}{\sqrt{1+y^2}}
\right)=\frac{mw^{m-1}}{(\sqrt{1-w^2})^{m-1}}+\frac{(m-1)w^{m+1}}{(\sqrt{1-w^2})^{m+1}}
=my^{m-1}+(m-1)y^{m+1}.$$
Substituting the latter expression to (\ref{poisbr1}) yields 
$$\{H_{d,f},H_{k,g}\}=\frac{(dy^{d-1}+(d-1)y^{d+1})y^kfg'-(ky^{k-1}+(k-1)y^{k+1})y^dgf'}{\sqrt{1+y^2}}.$$
This implies (\ref{poissform}). 
\end{proof}

\subsection{Case, when $\gamma$ is an interval. Proof of Theorems \ref{tlidf2}, \ref{tlidf}, \ref{tlidense3}} 
Recall that  for every $d\in\zz_{\geq0}$  
$\La_d$ denotes the vector space of functions 
of the type $H_{d,f}$, see (\ref{hdh}), 
where $f(s)$ runs through all the $C^{\infty}$-smooth 
functions in one variable. Let 
$$\pi_k:\oplus_{d=0}^{+\infty}\La_d\to\La_k$$
 denote the 
projection to the $k$-th component.  
\begin{proposition}  \label{cdense} One has 
\begin{equation}\{\La_0,\La_0\}=\La_1, \ \ \{\La_1,\La_1\}\subset\La_1,\label{01}\end{equation}
\begin{equation}\{\La_d,\La_k\}\subset\La_{d+k-1}\oplus\La_{d+k+1} \text{ whenever } 
(d,k)\neq(1,1),(0,0),\label{dk}\end{equation}
\begin{equation}\pi_{k+1}(\{\La_0,\La_k\})=\La_{k+1} \text{ for every } k\geq1.
\label{k+1}\end{equation}
\end{proposition}
\begin{proof}
Inclusion (\ref{dk}) and  the right inclusion in (\ref{01}) follow immediately from  (\ref{poissform}). 
Let us prove the left formula in (\ref{01}). One has 
\begin{equation}\{ H_{0,f}, H_{0,g}\}=H_{1,f'g-g'f},\label{h00}\end{equation}
by (\ref{poissform}). It is clear that each function $\eta(s)$ can be represented 
by an expression $f'g-g'f$, since the functions in question are defined on an 
interval. For example, one can take $f=\int_{s_0}^s\eta(\tau)d\tau$ and $g\equiv1$.   This proves the left formula in (\ref{01}). The proof of statement (\ref{k+1}) 
is analogous. \end{proof}

\begin{proof} {\bf of Theorem \ref{tlidf2}.} The Hamiltonians of the vector 
fields $v_f$ are the functions $-2H_{0,f}$. The Lie algebra $\gh$ 
generated by them coincides with $\oplus_{k=0}^{\infty}\La_k$, 
by Proposition \ref{cdense}. This proves Theorem \ref{tlidf2}.
\end{proof} 
\begin{proof} {\bf of Theorem \ref{tlidf} for $\gamma$ being an interval.} Each function from the direct sum $\gh=\oplus_{k=0}^{\infty}\La_k$ is  
 $\frac1{\sqrt{1+y^2}}$ times a polynomial in $y$ with 
coefficients depending on $s$. The latter polynomials include all the polynomials in $(s,y)$, 
which are $C^{\infty}$-dense in the space of $C^{\infty}$ functions in $(s,y)\in\gamma\times\rr$ 
(Weierstrass Theorem). 
Therefore, $\gh$ is also dense. Theorem \ref{tlidf} is proved in the case, when $\gamma$ is an interval. 
\end{proof}

\subsection{Case of closed curve. Proof of Theorem \ref{tlidf3}}

Let $H_{d,h}$ be functions  given by formula (\ref{hdh}); now  
$h(s)$ being $C^{\infty}$-smooth functions on the circle equipped with the natural length parameter $s$. 
Recall that we consider that its length 
is equal to $2\pi$, rescaling the metric by constant factor.  Let 
 $\La_d$ be the vector space of all the functions $H_{d,h}$. Set 
 $$\La_{d,0}:= \{ H_{d,h}\in\La_d \ | \ \int_0^{2\pi}h(s)ds=0\}.$$
For every $C^{\infty}$-smooth function $h:S^1=\rr\slash2\pi\zz\to\rr$ 
 set 
 $$\wh h:=\frac1{2\pi}\int_0^{2\pi}h(s)ds.$$

In the proof of  Theorem \ref{tlidf3} we use the following four propositions. 

\begin{proposition} For every $d,k\in\zz_{\geq0}$ and every pair of smooth functions $f(s)$ 
and $g(s)$ on the circle one has
$$\{ H_{d,f}, H_{k,g}\}=H_{d+k-1,h_{d+k-1}(s)}+H_{d+k+1,h_{d+k+1}(s)},$$
  \begin{equation} (d+k-2)\wh h_{d+k-1}=(d+k)\wh h_{d+k+1}.\label{rapav}\end{equation}
  \end{proposition}
\begin{proof} The first formula in (\ref{rapav}) holds with 
\begin{equation}h_{d+k-1}=(d+k)fg'-k(fg)', \ h_{d+k+1}=(d+k-2)fg'-(k-1)(fg)',\label{dkfg}\end{equation}
by (\ref{poissform}). This together with the fact that the derivative $(fg)'$ has zero 
average implies that the ratio of averages of the functions $h_{d+k\mp1}$ is equal to 
$\frac{d+k}{d+k-2}$ and proves (\ref{rapav}).
\end{proof}

It is clear that the Lie algebra $\gh$ is contained in the direct sum of the subspaces 
$\La_j$, by (\ref{poissform}). 
Recall that for every $j\in\zz_{\geq0}$ by $\pi_j$ 
we denote the projection of the latter direct sum to the $j$-th component 
$\La_j$. 

\begin{proposition} For every $d,k\in\zz_{\geq0}$  one has 
\begin{equation}\{\La_d,\La_k\}\subset\La_{d+k-1}\oplus\La_{d+k+1},
\label{osum}\end{equation}
\begin{equation}\pi_{d+k+1}(\{\La_d,\La_k\})=\pi_{d+k+1}(\{\La_{d,0},\La_{k,0}\})=\begin{cases} 0 \text{ if } d=k=1,\\
\La_{d+k+1}, \text{ if } d+k\neq2,\\
\La_{3,0}, \text{ if } \{ d,k\}=\{0,2\}.\end{cases}
\label{j=2}\end{equation}
In particular,
\begin{equation}\{\La_0,\La_0\}=\La_1.\label{laol}\end{equation} 
\end{proposition}
\begin{proof} For $d=k=1$ formula (\ref{j=2}) follows from (\ref{poissform}). For 
$d=0$, $k=2$ one has $\pi_3(\{\La_0,\La_2\})\subset\La_{3,0}$, by (\ref{rapav}): the left-hand 
side in (\ref{rapav}) vanishes, hence, $\wh h_{d+k+1}=\wh h_3=0$. 
Let us prove that in fact, the latter inclusion is equality and moreover, 
\begin{equation}\pi_3(\{\La_{0,0},\La_{2,0}\})=\La_{3,0}.\label{2=3}\end{equation}
 Indeed, for every two functions 
$f$ and $g$ on the circle one has 
$$\{ H_{0,f}, H_{2,g}\}=H_{1,-2f'g}+H_{3,-(fg)'},$$
by (\ref{poissform}). It is clear that every function $h$ on the circle with 
zero average is a derivative of some function $-f$ on the circle (we choose $f$ 
with zero average). 
Hence, $h=-f'=-(fg)'$ with $g\equiv1$. This already implies  the formula $\pi_3(\{\La_{0},\La_{2}\})=\La_{3,0}$, 
but not the above formula (\ref{2=3}): the function $g\equiv 1$ 
does not have zero average. To prove (\ref{2=3}), let us show that every function $f$ 
with zero average can be represented as a sum $\sum_{j=1}^4f_jg_j$ with 
$f_j$ and $g_j$ being smooth functions of zero average. Indeed, $f$ is the sum of 
  a linear combination $f_{=1}(s)=ae^{is}+\bar a e^{-is}$, $a\in\cc$, 
 and a Fourier series $f_{\geq2}$ containing only $e^{ins}$  with $|n|\geq2$. 
It is clear that $f_{=1}(s)=e^{3is}(e^{-3is}f_{=1}(s))$ 
and $f_{\geq2}=e^{is}(e^{-is}f_{\geq2})$, and the latter are products of two complex functions 
with zero average. Their real parts are obviously  sums of pairs of such products. Therefore, 
$f$ can be represented as a sum of four such products. This together with the above 
discussion implies that for every function $h$ with zero average 
the function $H_{3,h}$ is the $\pi_3$-projection of a sum of four Poisson brackets 
$\{ H_{0,f_j}, H_{2,g_j}\}$ with $f_j$, $g_j$ being of zero average. This proves (\ref{2=3}) and the third formula 
in (\ref{j=2}).

Let us now treat the remaining middle case: $d+k\neq2$. To do this, it suffices to 
show that every smooth function $h(s)$ on the circle can be 
represented as a  finite sum 
\begin{equation} h=\sum_{l=1}^N h_{d+k+1, l}, \ h_{d+k+1,l}:=(d-1)f_lg_l'-(k-1)f_l'g_l,\label{eqsolve}
\end{equation}
 see (\ref{rapav}) and (\ref{poissform}), where $f_l(s)$ and $g_l(s)$ are smooth 
functions on the circle with zero average. Moreover, 
it suffices to prove the same statement for complex-valued functions. Indeed, if (\ref{eqsolve}) 
holds for a complex function $h(s)$ and a finite collection of pairs $(f_l(s),g_l(s))$ of complex functions, 
$l=1,\dots,N$, then the similar equality holds for the function 
$\re h(s)$ and the collection of pairs $(\re f_l,\re g_l)$, $(-\im f_l, \im g_l)$ taken for all $l$. Let, say, $d\neq1$. 
Let us write a complex function $h(s)$ as a Fourier series 
$$h(s)=\sum_{n\in\zz}a_ne^{ins}.$$
Set 
$$f_1(s)=e^{is}, \ g_1(s)=\sum_{n\in\zz_{\neq1}}b_ne^{i(n-1)s}.$$
Set $h_{d+k+1, 1}:=(d-1)f_1g_1'-(k-1)f_1'g_1$. One has  
\begin{equation} h(s)-h_{d+k+1,1}(s)=a_1e^{is}+\sum_{n\neq1}(a_n-ib_n((d-1)(n-1)-(k-1)))e^{ins}.\label{anbnk}\end{equation}
We would like to make the above difference zero. 
For each individual $n\neq 1$ one can solve the equation 
$$a_n-ib_n((d-1)(n-1)-(k-1))=0$$
 in $b_n$, provided that $(d-1)(n-1)\neq k-1$, i.e., $n\neq n(d,k):=\frac{k-1}{d-1}+1$. Take $b_n$ found from  the above 
 equation for all $n\neq 1, n(d,k)$. They yield  a converging and $C^{\infty}$-smooth Fourier series 
$$g_1(s)=\sum_{n\neq 1, n(d,k)}b_ne^{i(n-1)s},$$
since so is $h(s)=\sum_{n\in\zz}a_ne^{ins}$ and $b_n=o(a_n)$, as $n\to\infty$. 
The corresponding function $h_{d+k+1,1}$, see (\ref{eqsolve}), satisfies the equality 
\begin{equation}h(s)-h_{d+k+1,1}(s)=a_1e^{is}+a_{n(d,k)}e^{isn(d,k)}.
\label{hdk1}\end{equation}
Now we set $f_2(s)=e^{(p+1)is}$ with some $p\in \zz\setminus\{0,-1,n(d,k)-1\}$, and we would like to find a function 
$$g_2(s)=c_{1}e^{-pis}+c_2e^{i(n(d,k)-(p+1))s}$$
such that $h=h_{d+k+1,1}+h_{d+k+1,2}$, see (\ref{eqsolve}). The latter equation is equivalent to the equation
\begin{equation}a_1e^{is}+a_{n(d,k)}e^{isn(d,k)}=(d-1)f_2(s)g_2'(s)-(k-1)f_2'(s)g_2(s),
\label{eqso}\end{equation}
by (\ref{hdk1}). Its right-hand side divided by $i$ equals $c_1(-p(d-1)-(p+1)(k-1))e^{is}+c_2((d-1)(n(d,k)-(p+1))-(k-1)(p+1))e^{isn(d,k)}$.
 Therefore, one can find constant coefficients $c_1$, $c_2$  in the  definition of the function $g_2$ such that equation (\ref{eqso}) holds, if 
 \begin{equation}(n(d,k)-(p+1))(d-1)-(p+1)(k-1)\neq0,   \ p(d-1)+(p+1)(k-1)\neq0.\label{2ineq}\end{equation}
 The left-hand sides of inequalities (\ref{2ineq}) are linear non-homogeneous functions in $p$ with coefficients at 
 $p$ being equal to $\mp(d+k-2)\neq0$. Hence, choosing appropriate $p\in\zz\setminus\{0,-1,n(d,k)-1\}$ 
 one can achieve that inequalities (\ref{2ineq}) hold and hence, equation (\ref{eqso}) can be solved in $c_1$, $c_2$. 
   Finally, we have solved equation (\ref{eqsolve}) with an arbitrary complex 
 function $h$ and $N=2$, in complex functions $f_l$, $g_l$, $l=1,2$ with zero averages. 
 This together with the above discussion 
finishes the proof of statement (\ref{j=2}). Statement (\ref{laol}) follows 
from the second statement in (\ref{j=2}).
\end{proof}

\begin{proposition} \label{pcont}
 The Lie algebra $\gh$ contains the direct sum 
$$\ggo:=\La_1\oplus(\oplus_{d\in 2\zz_{\geq0}}\La_d)\oplus(\oplus_{d\in2\zz_{\geq1}+1}\La_{d,0}).
$$
\end{proposition}
\begin{proof} The algebra $\gh$ contains $\La_0$ and $\La_1=\{\La_0,\La_0\}$, see 
(\ref{laol}). It also contains $\La_2$,  by the latter statement and 
since $\{\La_0,\La_1\}\subset\La_0\oplus\La_2$, see (\ref{osum}), 
and $\pi_2(\{\La_0,\La_1\})=\La_2$, by (\ref{j=2}). Hence, 
$\gh\supset\oplus_{j=0}^2\La_j$. Analogously one has $\gh\supset\La_{3,0}$, 
by the latter statement, and (\ref{osum}), (\ref{j=2}) applied to $(d,k)=(0,2)$. Hence, 
$\gh\supset(\oplus_{j=0}^2\La_j)\oplus\La_{3,0}$. 
One has $\gh\supset \La_4$, by the latter statement and (\ref{osum}), (\ref{j=2}) applied to 
$(d,k)=(0,3)$. Hence, $\gh\supset(\oplus_{0\leq j\leq4, j\neq3}\La_j)\oplus\La_{3,0}$. 
Let us show that $\gh\supset\La_{5,0}$. Indeed, the bracket $\{\La_0,\La_4\}$ is 
contained in $\La_3\oplus\La_5$, see (\ref{osum}). For every pair of functions 
$f$ and $g$ on the circle one has 
$$\{ H_{0,f}, H_{4,g}\}=H_{3,h_3}+H_{5,h_5},$$
 where the ratio of averages of the 
functions $h_3$ and $h_5$ is equal to 2, see (\ref{rapav}). Therefore, if the average 
of the function $h_5$ vanishes, then so does the average of the function $h_3$. 
This implies that the subspace of those elements in $\{\La_0,\La_4\}$ whose 
projections to $\La_5$ 
have zero averages coincides with the subspace with the analogous property 
for the projection $\pi_3$. Recall that $\pi_5(\{\La_0,\La_4\})=\La_5$, by  (\ref{j=2}). 
This together with the above statement and the inclusion $\La_{3,0}\subset\gh$
implies that  $\gh$ contains $\La_{5,0}$. Applying the above 
argument successively to Poisson brackets $\{\La_0,\La_n\}$, $n\geq5$, we get 
the statement of Proposition \ref{pcont}.
\end{proof}

\begin{proposition} \label{prepsum} The Lie algebra $\gh$ is the direct sum of the subspace 
$\ggo$ and a vector subspace $\Psi$ in 
$$\mcp=\{ \frac{P(y)}{\sqrt{1+y^2}} \ | \ P \text{ is an odd polynomial },  \ P'(0)=0\}.$$
The corresponding subspace $\sqrt{1+y^2}\Psi$ in the space of polynomials 
is generated by the polynomials 
\begin{equation} R_j(y):=jy^{2j-1}+(j-1)y^{2j+1}, \ j\in\nn, \ j\geq2.\label{rdy}\end{equation}
\end{proposition}
\begin{proof} The direct sum $\oplus_{j\geq0}\La_j\supset\gh$ is the direct 
sum of the spaces $\ggo$ and $\mcp$, by definition. 
This together with Proposition \ref{pcont} 
implies that $\gh$ is the direct sum of the subspace $\ggo$ and a subspace  
$\Psi\subset\mcp$.  Let us describe the subspace $\Psi$. To do thus, consider the projection 
$$\pi_{odd>1}:\oplus_{j\geq0}\La_j\to\oplus_{j\in2\zz_{\geq1}+1}\La_j.$$

{\bf Claim 1.} {\it The projection $\pi_{odd>1}\gh$ lies in $\gh$. It is spanned as a vector space 
over $\rr$ by the subspace $\La_{3,0}$  and some 
Poisson  brackets $\{ H_{d,a(s)}, H_{k,b(s)}\}$ with $d+k\geq4$ being even.  All the above brackets  
with all the functions $a(s)$, $b(s)$ with zero average lie in $\pi_{odd>1}\gh$.} 
\medskip

\begin{proof} The inclusion $\pi_{odd>1}\gh\subset\gh$ follows from the fact that $\pi_{odd>1}$ 
is the projection along the vector subspace 
$\La_1\oplus(\oplus_{j\in2\zz_{\geq0}}\La_j)\subset\ggo\subset\gh$. Each element of the Lie algebra 
$\gh$ is represented as a sum of a vector in $\La_0$ and a linear combination of Poisson brackets 
$\{ H_{d,a}, H_{k,b}\}$, by definition and (\ref{poissform}). The latter Poisson brackets lie in 
$\La_{d+k-1}\oplus\La_{d+k+1}$,  by (\ref{osum}), 
and thus, have components of the same parity $d+k+1(\mod 2)$. Note that if $d+k-1=1$, then 
the above bracket lies in $\La_1\oplus\La_{3,0}\subset\ggo\subset\gh$, by  (\ref{j=2}), and its 
$\pi_{odd>1}$-projection lies in $\La_{3,0}$. The two last statements together 
imply the second statement of the claim. 
If   $a(s)$ and $b(s)$ have zero average, then $H_{d,a}, 
H_{k,b}\in\ggo\subset\gh$, thus, $\{ H_{d,a}, H_{k,b}\}\in\gh$. Hence, the latter bracket lies in $\pi_{odd>1}\gh$, 
if $d+k$ is even and no less than 4. This  proves the claim.
\end{proof}

Taking projection $\pi_\Psi$ of a vector $w\in\gh$ to $\Psi$ consists of first taking its projection 
$$\pi_{odd>1}w=\sum_{j=1}^kH_{2j+1,f_j(s)}=\frac1{\sqrt{1+y^2}}\sum_{j=1}^kf_j(s)y^{2j+1}$$
and then replacing each $f_j(s)$ in the above right-hand side by its average $\widehat f_j$:  
$$\pi_\Psi w=\frac1{\sqrt{1+y^2}}P_w(y), \ P_w(y)=\sum_{j=1}^k\widehat f_jy^{2j+1}.$$
If $w=\{ H_{d,a},H_{k,b}\}$ with $d+k=2j\geq4$, then $P_w(y)=cR_j(y)$, see (\ref{rdy}), 
$c\in\rr$, which follows from 
   (\ref{rapav}). This together with Claim 1 implies that the vector space $\sqrt{1+y^2}\Psi$ 
   is contained in the vector space spanned by the polynomials $R_j$. Now it remains to prove 
   the converse:   each $R_j$ is contained in $\sqrt{1+y^2}\Psi$. To this end, we have to show that one can choose the 
   above functions $a$ and $b$ with zero average so that $P_w(y)\neq0$, i.e., so that 
   the above constant factor $c$ be non-zero. This statement is implied by 
   the second equality in (\ref{j=2}) and can be also proved directly as follows. 
   Let $b(s)$ be an arbitrary smooth non-constant function 
   on the circle with zero average. Set $a(s)=b'(s)$. Let $d+k=2j\geq4$. Then 
   $$w=\{ H_{d,a},H_{k,b}\}=\frac1{\sqrt{1+y^2}}(f_{j-1}(s)y^{2j-1}+f_j(s)y^{2j+1}),$$ 
    
$$f_{j-1}=2j(b')^2-k(b'b)',$$
 by (\ref{dkfg}). The function $f_{j-1}(s)$ has positive average, since so does its first term, while 
   its second term has zero average. Therefore, $P_w=cR_j$, $c>0$. Proposition 
\ref{prepsum} is proved.
\end{proof}

\begin{lemma} \label{dergen} 
The vector subspace generated by the polynomials $R_j$ from 
(\ref{rdy}) coincides with the space of odd polynomials $P(y)$ with $P'(0)=0$ such that 
the corresponding polynomial $\wt P(x)=x^{-\frac12}P(x^{\frac12})$ has vanishing derivative 
at $-1$. 
\end{lemma}
As it is shown below, the lemma is implied by the  following proposition. 
\begin{proposition} An odd polynomial 
\begin{equation}P_k(y)=\sum_{j=1}^{k}a_jy^{2j+1}\label{pky}\end{equation}
 is a linear combination of polynomials $R_\ell(y)$, see (\ref{rdy}), if and only if 
 \begin{equation}\sum_{j=1}^{k}(-1)^jja_j=0.\label{eqgen}\end{equation}
 \end{proposition}
 \begin{proof} The polynomials (\ref{rdy}) obviously satisfy (\ref{eqgen}). In 
 the space of odd polynomials 
 $P(y)$ with $P'(0)=0$ of any given degree $d\geq5$ equation (\ref{eqgen}) defines a hyperplane. 
 The polynomials (\ref{rdy}) of degree no greater than $d$ also generate a 
 hyperplane there. Hence, these two hyperplanes coincide. The proposition is proved.
  \end{proof}
  
 \begin{proof} {\bf of Lemma \ref{dergen}.} For every odd polynomial $P(y)$ as in 
 (\ref{pky}) one has 
 $$\wt P(x)=\sum_{j=1}^ka_jx^j.$$
 Hence, equation (\ref{eqgen}) is equivalent to the equation $\wt P'(-1)=0$. 
 Lemma \ref{dergen} is proved.
 \end{proof} 
 
 \begin{proof} {\bf of Theorem \ref{tlidf3}.} Theorem \ref{tlidf3}  follows from 
 Proposition \ref{prepsum} and Lemma \ref{dergen}.
 \end{proof}
 
 \subsection{Proof of Theorem \ref{tlidf} for closed curve}
Theorem \ref{tlidf} is implied by Theorem \ref{tlidf3} proved above and the following lemma. 
 \begin{lemma} \label{lemden} The Lie algebra $\ggg$, see (\ref{algg}), is $C^{\infty}$-dense in the space of 
 smooth functions on $T_{<1}\gamma\simeq\gamma_s\times\rr_y$. 
 \end{lemma}
 \begin{proof} Let us multiply each function from $\ggg$ by $\sqrt{1+y^2}$; 
 then each function becomes a polynomial in $y$ with coefficients being smooth 
 functions on a circle. All the polynomials in $y$ with coefficients as above 
 are $C^{\infty}$-dense in the space of functions in $(s,y)\in\gamma\times\rr$, by Weierstrass Theorem. 
  The polynomials realized by functions 
 from $\ggg$ in the above way are exactly the polynomials that are represented in unique way 
 as sums of at most four polynomials 
 of the following types:
 
 1) any polynomial of degree at most 2;
 
 2) any even polynomial containing only monomials of degree at least 4; 
 
 3) any odd polynomial of type $P(y;s)=\sum_{j=1}^ka_j(s)y^{2j+1}$ with 
 coefficients $a_j(s)$ being of zero average;
 
 4) any odd polynomial  of type $P(y)=\sum_{j=1}^kb_jy^{2j+1}$ with constant coefficients 
 $b_j$ and $\wt P'(-1)=0$.

 For the proof of Lemma \ref{lemden} it suffices to show that the odd 
 polynomials of type 4) are $C^{\infty}$-dense in the space of odd polynomials in $y$ with constant 
 coefficients and vanishing derivative at 0. 
 
Take an arbitrary odd polynomial of type $Q(y)=\sum_{j=1}^kb_jy^{2j+1}$. 
 The polynomial 
 $\wt Q(x):=x^{-\frac12}Q(x^{\frac12})=\sum_{j=1}^kb_jx^k$ 
 can be approximated in the  topology of uniform convergence with derivatives on segments 
 $[0,A]$ with $A$ arbitrarily large 
 by polynomials $\wt R(x)$ with $\wt R(0)=0$ and 
 $\wt R'(-1)=0$. Indeed, let us extend the 
 restriction $\wt Q|_{\{ x\geq0\}}$ to a $C^{\infty}$-smooth function on the semi-interval $[-1,+\infty)$ 
 with vanishing derivative  at $-1$. Thus extended function can be 
 approximated by polynomials  $\wt H_n(x)$. One can 
 normalize the above polynomials $\wt H_n$ to vanish at $0$ 
 and to  have zero  derivative at $-1$ by adding a small linear non-homogeneous function 
 $a_nx+b_n$. Then 
 the corresponding polynomials $H_n(y):=y\wt H_n(y^2)$ are of type 4) and approximate $Q(y)$. 
 This together with the above discussion proves Lemma \ref{lemden}. This finishes the proof of Theorem 
 \ref{tlidf} for curves.
 \end{proof}
 \subsection{Proof of Theorems \ref{tlidense}, \ref{tlidense2} and \ref{tlidense3} for $n=1$} 
 Theorems \ref{tlidense} and \ref{tlidense2} follow from Theorem \ref{tlidf}. 
 
 In the case, when $\gamma$ is a germ of curve (a local curve parametrized by an interval), 
 the bundle $T_{<1}\gamma$ is a contractible space identified with a rectangle in the coordinates $(s,w)$. 
 Therefore, each symplectic vector field on $T_{<1}\gamma$ is Hamiltonian. This together with 
 Theorem \ref{tlidense2} implies density of the Lie algebra generated by the fields $v_f$ in the Lie algebra  
 of symplectic vector fields. This proves Theorem \ref{tlidense3}.  
 
 \section{The Lie algebra in  higher dimensions. Proof of Proposition \ref{pcom} and Theorems \ref{tlidf1}, \ref{tlidf}, \ref{tlidense}}
 
 Here we consider the case, when $n\geq2$ (whenever the contrary is not specified). 
 In Subsection 3.1 we give a formula for Poisson brackets of functions $H_{k,\phi}$ with 
 $\phi\in S^k(\ttg)$ and prove Proposition \ref{pcom}. Then in Subsection 3.2 we find the Lie algebra generated by the 
 functions $H_f(x,w)$ and prove Theorem \ref{tlidf1}. In Subsection 3.3 
 we prove Theorems \ref{tlidf}, \ref{tlidense} and \ref{tlidense2}. 
 
 \subsection{Poisson brackets and their calculations for $n\geq1$. Proof of Proposition \ref{pcom}}
 The results of the present subsection are valid for all $n\in\nn$.

Let $\gamma$ be a $n$-dimensional Riemannian manifold, $x\in\gamma$. Fix some orthonormal 
coordinates $z=(z_1,\dots,z_n)$ on $T_x\gamma$. 
Recall that the {\it normal coordinates} centered at $x$ on a neighborhood $U=U(x)\subset\gamma$ 
is its parametrization by a neighborhood of the origin in 
$\rr^n_z\simeq T_x\gamma$ that is given 
by the exponential mapping: $\exp:T_x\gamma\to\gamma$. It is well-known that 
in thus constructed normal coordinates $z$ on $U$ the Riemannian metric has the same first jet at $x$, as the 
standard Euclidean metric $dz_1^2+\dots+dz_n^2$.  
\begin{proposition} \label{prosym} Fix arbitrary normal coordinates 
$(z_1,\dots,z_n)$ centered at $x_0\in\gamma$ on a neighborhood  $U=U(x_0)\subset\gamma$. 
For every $x\in U$ and every vector $w\in T_x\gamma$ let $(w_1,\dots,w_n)$ denote its components in the basis 
$\frac{\partial}{\partial z_j}$ in $T_x\gamma$. This yields global coordinates $(z_1,\dots,z_n; w_1,\dots,w_n)$ 
on $T\gamma|_U$. At the points of  the fiber $T_{x_0}\gamma\subset T\gamma$ the canonical symplectic form 
$\omega=d\alpha$ coincides with the standard symplectic form $\omega_e:=\sum_{j=1}^ndw_j\wedge dz_j$. 
\end{proposition}
\begin{proof} Let $\alpha$ and $\alpha_e:=\sum_{j=1}^nw_jdz_j$ denote the Liouville forms (\ref{liof}) 
on $T\gamma$ defined respectively by the metric under question and the standard 
Euclidean metric $dz_1^2+\dots+dz_n^2$. Their 1-jets  at each point $(x_0,w)$ of the fiber $T_{x_0}\gamma$ 
coincide, since both metrics have the same 1-jet at $x_0$. Therefore, 
$\omega=d\alpha=d\alpha_e=\omega_e$ on $T_{x_0}\gamma$.
\end{proof}
\begin{corollary} In the conditions of Proposition \ref{prosym} for every two smooth functions $F$ and $G$ 
on an open subset in $\pi^{-1}(U)\subset T\gamma$ 
their Poisson bracket at points of the fiber $T_{x_0}\gamma$ is equal to the standard Poisson bracket in the 
coordinates $(z,w)$: 
\begin{equation}\{ F,G\}=\frac{dF}{dw}\frac{dG}{dz}-\frac{dF}{dz}\frac{dG}{dw}  \ \ \ \text{ on } \  T_{x_0}\gamma.
\label{poissn}\end{equation}
\end{corollary}

\begin{proof} {\bf of Proposition \ref{pcom}.} Fix a point $x_0\in\gamma$ and a system of normal coordinates 
$(z_1,\dots,z_n)$ centered at $x_0$. Let $(w_1,\dots,w_n)$ be the corresponding coordinates on tangent spaces 
to $\gamma$ introduced above. For every $f,g\in C^{\infty}(\gamma)$ one has 
$$\{ f(z)\sqrt{1-||w||^2}, g(z)\sqrt{1-||w||^2}\}|_{z=0}$$
\begin{equation}=-\sqrt{1-||w||^2}
\sum_j\left(f\frac{\partial g}{\partial z_j}-g\frac{\partial f}{\partial z_j}\right)\frac{w_j}{\sqrt{1-||w||^2}}\label{poissfo}
\end{equation}
$$=(gdf-fdg)(w)=(2gdf-d(gf))(w):$$
the derivative in $z$ of the function $\sqrt{1-||w||^2}$ vanishes at the points of the fiber $\{ z=0\}$, since 
the metric in question has the same first jet at $0$, as the Euclidean metric (normality of coordinates). 
Thus, the latter Poisson bracket is the function on $T_{<1}\gamma$ given by a 1-form. 

{\bf Claim 1.} {\it The Poisson  bracket (\ref{poissn}) 
of any two 1-forms $\eta_1(w)$, $\eta_2(w)$ considered as functions on $T\gamma$ is also a 1-form. The dual 
Poisson bracket on vector fields on $\gamma$ induced via the duality isomorphism $T^*\gamma\to T\gamma$ 
given by the metric is the usual Lie bracket.}

\begin{proof} The result of taking Poisson bracket (\ref{poissn}) of two 1-forms in normal chart at the 
fiber $T_{x_0}\gamma$ is obviously a linear functional of $w$. This is true for every $x_0\in\gamma$ and 
the corresponding normal chart. In more detail, let $\eta_i(w)=\sum_{j=1}^na_{ij}(z)w_j$. Then 
\begin{equation}\{\eta_1(w),\eta_2(w)\}|_{x=x_0}=\sum_{s=1}^n\left(a_{1s}(0)
\frac{\partial a_{2j}}{\partial z_s}(0)-a_{2s}(0)
\frac{\partial a_{1j}}{\partial z_s}(0)\right)w_j.\label{lieformyi}\end{equation}
Thus, the bracket is linear on fibers, and hence,  given by a 1-form. 
The fact that the duality given by the metric transforms the Poisson bracket on 1-forms to the Lie bracket 
on the vector fields follows from (\ref{lieformyi}) and coincidence of   1-jets 
of the given metric and the Euclidean metric at $x_0$.\end{proof}

Claim 1 already implies that the space of functions on $T_{<1}\gamma$ given by 1-forms is a Lie algebra 
dual to the usual Lie algebra of vector fields. 

Now it remains to show that each 1-form is a  
sum of at most $2n+1$ Poisson brackets (\ref{poissfo}). To do this, we use the next proposition.

\begin{proposition} \label{claim2} On every $C^{\infty}$-smooth $n$-dimensional 
manifold $\gamma$ 
each smooth 1-form $\eta$  can be represented as a finite sum $\sum_{\ell=1}^{2n} g_\ell df_\ell$, where $f_\ell$, $g_\ell$ are smooth functions on $\gamma$.
\end{proposition}

\begin{proof} Consider $\gamma$ as an embedded submanifold in 
$\rr^{2n}_{s_1,\dots,s_{2n}}$ (the Whitney Theorem). We would like to show that the form $\eta$ 
can be written 
as the restriction to $\gamma$ of some differential 1-form $\wt\eta$ on $\rr^{2n}$. To do this,  fix a tubular neighborhood $\Gamma^\delta\subset\rr^{2n}$ of the submanifold $\gamma$ given by the 
Tubular Neighborhood Theorem. Here $\delta=\delta(x)>0$ is a smooth function, 
\begin{equation}\Gamma^{\delta}:=\{ p\in\rr^{2n} \ | \ \dist(p,x)<\delta(x) \text{ for some } x\in\gamma\}.\label{tubn}
\end{equation}
 Each point $p\in\Gamma^{\delta}$ has the unique closest point $x=\pi_{\delta}(p)\in\gamma$, 
 and the projection $\pi_{\delta}:\Gamma^{\delta}\to\gamma$ is a submersion. 
  The projection $\pi_\delta$ allows to 
extend the form $\eta$ to the pullback form $\hat\eta:=\pi_{\delta}^*\eta$ on $\Gamma^\delta$, which coincides with 
$\eta$ on $\gamma$. Take now an arbitrary smooth bump function $\beta$ on $\Gamma^\delta$ that is identically 
equal to $1$ on $\gamma$ and vanishes on a neighborhood of the boundary $\partial\Gamma^\delta$. 
The 1-form $\wt\eta:=\beta\hat\eta$ extended by zero outside the tubular neighborhood $\Gamma^{\delta}$ is a global 
smooth 1-form on $\rr^{2d}$ whose restriction to $\gamma$ coincides with $\eta$. Taking its coordinate representation 
$\wt\eta=\sum_{\ell=1}^{2n}\wt\eta_\ell ds_\ell$ and putting $g_\ell=\wt\eta_\ell$, $f_\ell=s_\ell$ yields the statement 
of the proposition.
\end{proof}

Let now $\eta$ be an arbitrary smooth 1-form on $\gamma$. Let $f_\ell$, $g_\ell$ be the same, as in Proposition 
\ref{claim2}. 
Set $g_{2n+1}=1$, $f_{2n+1}=\sum_{\ell=1}^{2n}g_{\ell} f_{\ell}$. Then 
$$2\eta=2\sum_{\ell=1}^{2n}g_\ell d f_\ell=\sum_{\ell=1}^{2n}(2g_\ell df_\ell-d(g_{\ell}f_{\ell}))+(2g_{2n+1}df_{2n+1}-d(g_{2n+1}f_{2n+1})).$$
Therefore, $2\eta$, and hence $\eta$ can be presented as a sum of at most $2n+1$ Poisson brackets (\ref{poissfo}). 
This finished the proof of Proposition \ref{pcom}.
\end{proof}

For every $x\in\gamma$ let $\sym^k(T^*_x\gamma)$ denote the space of symmetric $k$-linear forms on $T_x\gamma$:  
each of them sends a collection of $k$ vectors in $T_x\gamma$ to a real number and is symmetric under permutations of 
vectors. The union of the latter spaces through all $x\in\gamma$ is a vector bundle of tensors. 
The space of its sections will be 
denoted  by $\sym^k(T^*\gamma)$. Similarly, by $S^k(T^*_x\gamma)$ we will denote the space of degree $k$ 
homogeneous polynomials as functions on the vector space $T_x\gamma$. The space $S^k(\ttg)$ introduced 
before is the space of sections of the vector bundle over $\gamma$ with fibers $S^k(T^*_x\gamma)$. 
It is well-known that the mapping $\chi_{sym}:\sym^k(T^*_x\gamma)\to S^k(T^*_x\gamma)$ sending a $k$-linear form $\Phi(y_1,\dots,y_k)$ to the polynomial 
 $\phi(y):=\Phi(y,\dots,y)$ is an isomorphism, which induces a section space isomorphism $\sym^k(\ttg)\to S^k(\ttg)$. The notion of 
 covariant derivative of a tensor bundle section $\Phi\in\sym^k(T^*\gamma)$  is well-known. 
 For every $x\in\gamma$, $\nu\in T_x\gamma$, $\phi\in S^k(\ttg)$  
 the {\it covariant derivative} $\nabla_\nu\phi\in S^k(T^*_x\gamma)$  along a vector $\nu\in T_x\gamma$ is 
 \begin{equation}\nabla_\nu\phi:=\chi_{sym}(\nabla_\nu(\chi_{sym}^{-1}\phi)).\label{chisym}\end{equation} 
\begin{remark} \label{rkrist} For every $\phi\in\skt$ one has 
\begin{equation} V_k(\phi)(y):=(\nabla_y\phi)(y)\in S^{k+1}(\ttg).\label{vkfy}\end{equation} 
 Let $z=(z_1,\dots,z_n)$ be normal coordinates on $\gamma$ centered at $x$,  and let 
$(w_1,\dots,w_n)$ be the corresponding coordinates on the fibers of the bundle $T\gamma$ in the basis $\frac{\partial}{\partial z_1},\dots,\frac{\partial}{\partial z_n}$. For every $\nu\in T_x\gamma$  and every $\phi\in S^k(\ttg)$ considered 
as a polynomial in $w$ with coefficients depending on $z$  
the covariant derivative (\ref{chisym}) is obtained from the polynomial $\phi$ by 
replacing its coefficients  by their derivatives along the vector $\nu$. This follows 
from definition, since in normal coordinates the Christoffel symbols vanish. 
\end{remark}
\begin{convention}  Everywhere below for every $\phi\in\skt$, $x\in\gamma$ and $\nu\in T_x\gamma$ 
by $\frac{d_y\phi(y)}{d\nu}$ we denote {\bf the derivative of the polynomial  $\phi(y)$} considered 
as a function of $y\in T_x\gamma$ 
along the vector $\nu$ (treating $\nu$ as a tangent vector to $T_x\gamma$ at $0$, identifying the vector space $T_x\gamma$ with its tangent spaces at 
all its points by translations). 
For every function $f\in C^{\infty}(\gamma)$ by $\frac{d_y\phi(y)}{d\nabla f}$ we 
denote the above derivative calculated for $\nu=\nabla f(x)$ at each point $x\in\gamma$. Note that the  
latter derivative lies in $S^{k-1}(\ttg)$, by definition.  
\end{convention}

\begin{proposition} \label{propoisf}  
 For every $k,m\in\zz_{\geq0}$ and every $\phi\in S^k(T^*\gamma)$, $f\in S^m(T^*\gamma)$
  the Poisson bracket of the corresponding 
functions $H_{k,\phi}$ and $H_{m,f}$ from (\ref{hkhf}) (see Convention \ref{conve}) lies in 
$\La_{k+m-1}\oplus\La_{k+m+1}$. In the case, 
when $m=0$, i.e., $f=f(x)$ is a function of $x$, one has 
\begin{equation}\{ H_{k,\phi},H_{0,f}\}=\frac{\frac{d_y\phi(y)}{d\nabla f}+(\nabla_y(f\phi))(y)+
(k-2)\phi(y)<y,\nabla f>}{\sqrt{1+||y||^2}}.\label{poishf}\end{equation}
Here $\nabla f$ is the gradient of the function $f(x)$ with respect to the metric of $\gamma$, and 
$\nabla_y(f\phi)$ is the above-defined covariant derivative of the form $f\phi\in\skt$ along the vector $y$, see (\ref{chisym}), (\ref{vkfy}). 
\end{proposition}
\begin{proof} Let us prove the first statement of the proposition. Fix an arbitrary $x_0\in\gamma$.  It suffices 
to show that the restriction to $T_{x_0}\gamma$ of the Poisson bracket is $(\sqrt{1+||y||^2})^{-1}$ times 
a sum of two homogeneous polynomials in $y$ of degrees $k+m-1$ and $k+m+1$. Let us pass back to 
the initial unit ball bundle $T_{<1}\gamma$. Fix some normal coordinates $(z_1,\dots,z_n)$ centered at $z_0$ 
and the corresponding coordinates $w=(w_1,\dots,w_n)$ on the fibers. One has 
\begin{equation}H_{k,\phi}=\frac{\phi(w)}{(\sqrt{1-||w||^2})^{k-1}};\label{hkphi}\end{equation}
here the squared norm $||w||^2$ is given by the metric and depends on the point $z\in\gamma$. 
This follows by definition and since 
\begin{equation} ||y||^2=\frac{||w||^2}{1-||w||^2}, \ ||y||^2+1=(1-||w||^2)^{-1}.
\label{wandy}\end{equation}
Let us calculate  the Poisson bracket $\{ H_{k,\phi},H_{m,f}\}$ at points in $T_{z_0}\gamma$ 
 by formula (\ref{poissn}). The partial derivative 
$\frac{\partial H_{k,\phi}}{dz_j}$ is $(\sqrt{1-||w||^2})^{-(k-1)}$ times a degree $k$ homogeneous polynomial in $w$  obtained from the  polynomial $\phi(w)$ by replacing its coefficients by their partial derivatives in $z_j$. 
This follows from definition and the fact that $\sqrt{1-||w||^2}$ has zero derivatives in $z_j$ at points of the fiber $T_{x_0}\gamma$. The latter fact follows from normality of the coordinates $z_j$. 
The partial derivative 
$\frac{\partial H_{k,\phi}}{\partial w_j}$ is a sum of a degree $k-1$ homogeneous polynomial in $w$ divided by 
$(\sqrt{1-||w||^2})^{k-1}$ and a degree $k+1$ homogeneous polynomial in $w$ divided by $(\sqrt{1-||w||^2})^{k+1}$ 
(Leibniz rule). Finally, the expression (\ref{poissn}) for the above Poisson bracket is a sum of 
terms of the two following types: 1) a degree $k+m-1$ homogeneous polynomial in $w$ divided by 
$(\sqrt{1-||w||^2})^{k+m-2}$; 2) a degree $k+m+1$ homogeneous polynomial in $w$ divided by 
$(\sqrt{1-||w||^2})^{k+m}$. The terms of types 1) and 2) are homogeneous polynomials in $y$ of degrees 
respectively $k+m\pm1$ divided by $\sqrt{1+||y||^2}$, by (\ref{wandy}). This implies the first statement of the 
proposition. 

Let now $f=f(x)$ be a function of $x\in\gamma$. Let us prove formula (\ref{poishf}). For $z=0$ (i.e., $x=x_0$) one has  
$$\frac{\partial H_{k,\phi}}{\partial w_j}=(\sqrt{1-||w||^2})^{1-k}\frac{\partial}{\partial w_j}\phi(w)+\frac{(k-1)\phi(w)w_j}{(\sqrt{1-||w||^2})^{k+1}},$$
$$\frac{\partial H_{0,f}}{\partial w_j}=-\frac{w_j}{\sqrt{1-||w||^2}}f, \ \ \frac{\partial H_{0,f}}{\partial z_j}=\sqrt{1-||w||^2}\frac{\partial f}{\partial z_j},$$
\begin{equation}\sum_{j=1}^n\frac{\partial H_{k,\phi}}{\partial w_j}\frac{\partial H_{0,f}}{\partial z_j}=
(\sqrt{1-||w||^2})^{2-k}\frac{d_w\phi(w)}{d\nabla f}+\frac{(k-1)\phi(w)<\nabla f,w>}{(\sqrt{1-||w||^2})^{k}},\label{hkphi}
\end{equation}
\begin{equation}
\sum_{j=1}^n\frac{\partial H_{k,\phi}}{\partial z_j}\frac{\partial H_{0,f}}{\partial w_j}=
-\frac f{(\sqrt{1-||w||^2})^k}\sum_{j=1}^nw_j\frac{\partial \phi(w)}{\partial z_j}=-\frac{(\nabla_w\phi)(w)f}{(\sqrt{1-||w||^2})^k},\label{nnabla}\end{equation}
which follows from Remark \ref{rkrist}. 
 Substituting formulas (\ref{hkphi}), (\ref{nnabla})  
 and the formula $<\nabla f,w>\phi(w)+(\nabla_w\phi)(w)f=(\nabla_w(f\phi))(w)$ to (\ref{poissn})  yields 
$$\{ H_{k,\phi}, H_{0,f}\}=\frac{\frac{d_w\phi(w)}{d\nabla f}}{(\sqrt{1-||w||^2})^{k-2}}$$
$$+
\frac{(k-2)\phi(w)<\nabla f,w>}{(\sqrt{1-||w||^2})^{k}}+\frac1{(\sqrt{1-||w||^2})^{k}}(\nabla_w(f\phi))(w),$$
which is equivalent to (\ref{poishf}). Proposition \ref{propoisf} is proved.
\end{proof}
\begin{remark} Formula (\ref{poishf}) remains valid for $n=1$ and yields 
 formula (\ref{poissform}) from Subsection 2.1 in the case, when $k=0$ in  (\ref{poissform}). 
  Indeed, in this case $\phi=h(s)y^d$, $f=f(s)$. The corresponding first term in the 
 numerator in (\ref{poishf}) is equal to $dhf'y^{d-1}$. The second term equals $(hf)'y^{d+1}$. The third term 
 equals $(d-2)hf'y^{k+1}$. Thus, the numerator is equal to $dhf'y^{d-1}+((d-1)hf'+fh')y^{d+1}$. 
 This together with (\ref{poishf}) yields (\ref{poissform}).
 \end{remark}

\subsection{The Lie algebra. Proof of Theorem \ref{tlidf1}}
\begin{proposition} \label{pcontin} The Lie algebra $\gh$ generated by functions $H_{k,\phi}$ for all $k\in\zz_{\geq0}$ 
and $\phi\in S^k(T^*\gamma)$ with respect to the Poisson bracket 
is contained  in $\La:=\oplus_{k\geq0}\La_k$, and the latter direct sum is a Lie algebra.
\end{proposition}
Proposition \ref{pcontin} follows from Proposition \ref{propoisf}. 

In what follows by $\pi_k$ we denote the projection $\pi_k:\La\to\La_k$. 

We will deal with the linear operators $V_k$ given by (\ref{vkfy}): 
$$V_k:S^k(T^*\gamma)\mapsto S^{k+1}(T^*\gamma), \ \ \phi(y)\mapsto(\nabla_y\phi)(y).$$
\begin{proposition} \label{pikp1} One has 
\begin{equation}\pi_{k+1}\{ \La_k,\La_0\}=\begin{cases}\La_{k+1} \text { for } k\neq2,\\
\frac1{\sqrt{1+||y||^2}}V_2(S^2(T^*\gamma)) \text{ for } k=2.\end{cases}\label{pik+1}\end{equation}
\end{proposition}
\begin{proof} The higher order part of a Poisson bracket $\{ H_{k,\phi}, H_{0,f}\}$ is equal to 
\begin{equation}
\{ H_{k,\phi}, H_{0,f}\}_{k+1}=\frac1{\sqrt{1+||y||^2}}((k-2)\phi(y)<\nabla f,y>+(\nabla_y(\phi f))(y)),\label{hk0f}
\end{equation}
by (\ref{poishf}). If either $f\equiv 1$, or $k=2$, then the above expression in the brackets is reduced to 
$(\nabla_y(\phi f))(y)=V_k(\phi f)(y)$. Moreover, if $f\equiv1$, then the whole numerator in (\ref{poishf}) is reduced to $V_k(\phi f)(y)$. 
Therefore, each element in $\frac1{\sqrt{1+||y||^2}}V_k(S^k(T\gamma))$ is realized by a Poisson bracket, 
and hence, 
\begin{equation} \frac1{\sqrt{1+||y||^2}}V_k(S^k(T\gamma))\subset\{\La_k,\La_0\}.\label{inclbr}\end{equation}  
In particular,  this implies the statement of the proposition for $k=2$. To treate the case, when $k\neq2$, we use 
the following proposition. 
\begin{proposition} \label{pronaf} For every $k\in\zz_{\geq1}$ the vector subspace in $S^{k+1}(T^*\gamma)$ 
 generated by all the products $\phi(y)<\nabla f,y>=\phi(y)\frac{df}{dy}$ with $\phi\in S^k(T^*\gamma)$ 
 and $f\in C^{\infty}(\gamma)$  coincides with all of $S^{k+1}(T^*\gamma)$. Moreover, each element in $S^{k+1}(\ttg)$
 can be represented as a sum of at most $2n$ products as above. 
 \end{proposition}
 \begin{proof} Consider $\gamma$ as an embedded submanifold in 
$\rr^{2n}_{s_1,\dots,s_{2n}}$ (the Whitney Theorem); the embedding needs not be isometric. 

{\bf Claim 3.} {\it Each $\phi\in S^{k+1}(\ttg)$ is the restriction to $T\gamma$ of some $h\in S^{k+1}(T^*\rr^{2n})$.} 

\begin{proof} Let $\delta(x)>0$ be a smooth function on $\gamma$ with $\delta_{sup}:=\sup\delta\ll1$ that 
defines a tubular neighborhood $\Gamma^\delta\subset\rr^{2n}$ 
of the submanifold $\gamma$, see (\ref{tubn}). Let 
$\pi_{\delta}:\Gamma^{\delta}\to\gamma$ be the projection, which is a submersion. Let us extend $\phi$ to a form  $\wt\phi:=\pi_\delta^*\phi\in S^{k+1}(\Gamma^\delta)$ as the projection pullback. Let $\psi:\Gamma^\delta\to\rr$ 
be a bump function that is identically equal to 1 on a neighborhood of the submanifold $\gamma$ and 
vanishes on a neighborhood of the boundary $\partial\Gamma^\delta$. Then the form $h:=\psi\wt\phi$ extended 
as zero outside the tubular neighborhood $\Gamma^\delta$ becomes a well-defined form $h\in S^{k+1}(T^*\rr^{2n})$ 
whose restriction to $T\gamma$ coincides with $\phi$. 
\end{proof}

Let us consider the standard trivialization of the tangent bundle $T\rr^{2n}$ given by translations to $T_0\rr^{2n}$. 
Let $y=(y_1,\dots,y_{2n})$ be the corresponding coordinates on the fibers. 
Every $h\in S^{k+1}(T^*\rr^{2n})$ is a homogeneous polynomial in $y$ with coefficients being $C^{\infty}$-smooth functions 
on $\rr^{2n}$. Therefore, $h$ can be decomposed as 
$$h(y)=\sum_{j=1}^{2n}\phi_j(y)y_j=\sum_{j=1}^{2n}\phi_j(y)ds_j(y), \ \phi_j\in S^k(T^*\rr^{2n}).$$
The restriction of the latter decomposition to $T\gamma$ together with Claim 3 yield 
 the second (and hence, the first) statement of Proposition \ref{pronaf}.
 \end{proof}
The statement of Proposition \ref{pikp1} for $k\neq2$ follows from formula (\ref{hk0f}), statement (\ref{inclbr}) 
and Proposition \ref{pronaf}.
\end{proof}    
\begin{corollary} \label{coroif} The  sum $\La_0\oplus\La_1\oplus\La_2$ and the subspace  
$\frac1{\sqrt{1+||y||^2}}V_2(S^2(\ttg))$ in $\La_3$ are contained in the Lie algebra $\gh$. 
If $\La_3\subset\gh$, then $\gh=\oplus_{k\geq0}\La_k$. 
\end{corollary}
\begin{proof} The first statement of the corollary 
follows from definition and the statement of Proposition \ref{pikp1} for $k=0,1,2$. Its second 
statement follows from the statement of Proposition \ref{pikp1} for $k\geq3$. 
\end{proof} 

Now for the proof of  Theorems \ref{tlidf1} it suffices to show that $\La_3\subset\gh$. As it is shown 
below, this is implied by formula (\ref{poishf}) and the following lemma. 
   
   \begin{lemma} \label{pk=4} Let $n\geq2$. For every $k\in\nn$ consider two $\rr$-linear mappings 
   \begin{equation}
   G^+_k:S^k(\ttg)\otimes_\rr C^{\infty}(\gamma)\to S^{k+1}(\ttg), \ \phi\otimes_\rr f\mapsto \phi(y)\frac{df}{dy},
   \label{degk+}\end{equation}
   \begin{equation}G^-_k:S^k(\ttg)\otimes_\rr C^{\infty}(\gamma)\to S^{k-1}(\ttg), \  \phi\otimes_\rr f\mapsto \frac{d_y\phi(y)}{d\nabla f}.\label{degk-}\end{equation}
   (In (\ref{degk+}) for every $x\in\gamma$ and $y\in T_x\gamma$ the derivative 
   $\frac{df}{dy}=\frac{df}{dy}(x)$ means the derivative of the function $f$ at $x$ along the vector $y$.) 
   For every even $k$, e.g., $k=4$, one has $G^-_k(\ker G^+_k)=S^{k-1}(\ttg)$: more precisely, there exists a 
   $d_{k,n}\in\nn$, $d_{k,n}<(4n)^{k+1}$, such that for every 
   $h\in S^{k-1}(\ttg)$ there exists a  collection of $d_{k,n}$ pairs   
   $$ (\phi_j,f_j), \ \phi_j\in S^k(\ttg), \ f_j\in C^{\infty}(\gamma), \ j=1,\dots,d_{k,n},$$
    such that 
   \begin{equation} \sum_j\phi_j(y)\frac{df}{dy}\equiv 0, \ \ \sum_j\frac{d_y\phi_j(y)}{d\nabla f_j}=h.\label{phijfh}
   \end{equation}
   \end{lemma}
   
   Lemma \ref{pk=4} is proved below. In its proof we deal with $\gamma$ as an embedded submanifold in $\rr^{N}_{(s_1,\dots,s_{N})}$ (the Whitney Theorem, in which one can take $N=2n$), equipped 
   with its intrinsic Riemannian metric (not coinciding with the restriction to $T\gamma$ of the standard Euclidean metric). 
  Let us trivialize the tangent bundle $T\rr^{N}$ by translation to the origin
   and denote by $y=(y_1,\dots,y_{N})$ the corresponding 
   coordinates on the tangent spaces. 
Let $\mcp^k=\mcp^k(T^*\rr^{N})\subset S^k(T^*\rr^{N})$ be the subspace of degree $k$ homogeneous 
polynomials in $y$ with constant coefficients. 
Let $\mathbb L\subset C^{\infty}(\rr^{N})$ denote the $N$-dimensional vector subspace over $\rr$ 
generated by the coordinate functions $s_1,\dots,s_{N}$. 
\begin{remark} \label{rrestr} The operator $G^+_k$ given by formula (\ref{degk+}) extends as a well-defined linear operator 
$S^k(T^*\rr^N)\otimes_\rr C^{\infty}(\rr^N)\to S^{k+1}(T^*\rr^N)$ by 
the same formula, which will be also denoted by $G^+_k$. 
For every $h\in S^k(T^*\rr^N)\otimes_\rr C^{\infty}(\rr^N)$ the restriction $(G^+_kh)|_{T\gamma}$ coinsides with the image of the 
restriction $h|_{T\gamma}\in S^k(\ttg)\otimes_\rr C^{\infty}(\gamma)$ under the operator $G^+_k$ acting on $S^k(\ttg)\otimes_\rr C^{\infty}(\gamma)$. In particular, 
the restrictions to $T\gamma$ of elements of the kernel of the operator $G^+_k$ in the space $S^k(T^*\rr^N)\otimes_\rr C^{\infty}(\rr^N)$ are contained 
in the kernel of the operator $G^+_k$ acting on 
$S^k(\ttg)\otimes_\rr C^{\infty}(\gamma)$. 
\end{remark}

Step 1: finding basis of the kernel of $G^+_k$ in $\mcp^k\otimes\lll:=\mcp^k\otimes_\rr\lll$.
\begin{proposition} \label{ppbasis} 
For every $k\in\nn$  the kernel $\ker^k_{pol}$ of the restriction to $\mcp^k\otimes\mathbb L$ of the 
linear operator $G^+_k$ is the vector space with the basis 
\begin{equation}Q_{m,i,j}:=y^m(y_i\otimes s_{j}-y_{j}\otimes s_i), \ m=(m_1,\dots,m_{N})\in\zz_{\geq0}^{N}, 
\label{qmij}\end{equation} 
$$|m|:=\sum_\ell m_\ell=k-1, \ \  i,j=1,\dots,N, \ \ j>i,$$
  $$j\geq\max\ind(m):=\max\{ \ell \ | \ m_\ell>0\}.$$
\end{proposition}
\begin{proof} It is obvious that $G^+_kQ_{m,i,j}=0$. Let us show that the elements 
$Q_{m,i,j}$ are linearly independent. Indeed, 
if there were a linear dependence, then there would be a linear dependence between 
some $Q_{m,i,j}$, for which the corresponding monomial $P:=y^my_iy_j$ is the same: 
a fixed monomial. All the $Q_{m,\ell,\mu}$ corresponding to $P$ have  $\mu=j$, by definition, while $\ell$ 
runs through those indices less than $j$, for which $P$ contains $y_\ell$. 
It is clear that these $Q_{m,\ell,j}$ are linearly independent, -- a contradiction.  

Fix an arbitrary $Q\in\ker^k_{pol}$: 
\begin{equation} Q=\sum_{j,m}c_{j,m}y^{m}\otimes s_j; \ |m|=k.\label{qlinc}\end{equation}
 Let us show that $Q$ is a linear combination of elements $Q_{m,i,j}$. 
 In the proof we use the formula
 \begin{equation}G^+_k(y^m\otimes s_j)=y_jy^m \text{ for every } m\in\zz_{\geq0}^N, \ |m|=k,\label{gkjp}\end{equation} 
 which follows from definition. 
Without loss of generality we can and will consider that $i\geq\ell_0:=\max\ind(m)$ for every $(i,m)$ with 
$c_{i,m}\neq0$. Indeed, if the latter inequality does not hold for some $(i,m)$, one can achieve 
it by  adding 
$$c_{i,m}y^my_{\ell_0}^{-1}(y_i\otimes s_{\ell_0}-y_{\ell_0}\otimes s_i)=c_{i,m}Q_{m',i,\ell_0}$$ 
 to $Q$. This operation  
kills the $(i,m)$-th term and replaces it by an $(\ell_0,m')$-th term. Here $m'$ is obtained from $m$ by replacing $m_{\ell_0}$, $m_i$ by $m_{\ell_0}-1$ and 
$m_i+1$ respectively. Let us show that $Q=0$. Indeed, 
$G^+_k(c_{i,m}y^m\otimes s_i)=c_{i,m}y_iy^m$ should cancel out with similar monomials  
coming from other $(j,\wt m)\neq(i,m)$, since $G^+_k Q=0$. 
Therefore, there exist $j\neq i$ and $\wt m\neq m$ for which 
$c_{j,\wt m}\neq0$ and 
$y_jy^{\wt m}=y_iy^m$. Hence, $y^m$ is divisible by $y_j$ and $j\leq\ell_0=\max\ind(m)$. 
On the other hand,  $j\geq\max\ind(\wt m)$, by the above assumption on all the monomials in $Q$, and $y^{\wt m}$ 
is divisible by $y_i$. Hence, $\ell_0\geq j>i\geq\ell_0$, since $j\neq i$,  a contradiction. Therefore, $Q=0$. 
The proposition is proved.
\end{proof} 

Step 2: proof of an Euclidean homogeneous version of Lemma \ref{pk=4}. In what follows by $\mcp^k_0\subset\mcp^k$ we denote the 
subspace of polynomials with zero average along the unit sphere. 
\begin{lemma} \label{ppeucl} Let $N\geq2$. Consider the mapping $g^-_k:\mcp^k\otimes\lll\to\mcp^{k-1}$ acting by the formula 
\begin{equation} g^-_k:P(y)\otimes f\mapsto\frac{dP(y)}{d\nabla f};\  f=\sum_{j=1}^Nc_js_j, \ c_j=const,
\ \nabla f=(c_1,\dots,c_N), \label{graeq}\end{equation}
i.e., the above gradient is taken with respect to the standard Euclidean metric. 
One has $g^-_k(\ker^k_{pol})=\mcp^{k-1}_0$. The latter image coincides with all of $\mcp^{k-1}$, if and only if $k$ is even.
\end{lemma}
In the proof of Lemma \ref{ppeucl} we will use the following equivariance and invariance  properties of the 
operators $G^+_k$, $g^-_k$ and the
kernel $\ker^k_{pol}$  under the actions of $\gl_N(\rr)$ and $\oo(N)$ on $\mcp^k$ and $\mcp^k\otimes\lll$: \ \ 
$$H(P):=P\circ H, \ \ H(P\otimes s):=(P\circ H)\otimes(s\circ H)$$
$$\text{ for every } H\in\gl_N(\rr), \ P\in\mcp^k, \ s\in\lll.$$
\begin{proposition} \label{pequivar} 
1) The restriction of the mapping $G^+_k$ to $\mcp^k\otimes\lll$ is equivariant under the 
action of the linear group $\gl_N(\rr)$ on the image and the preimage: 
$$G^+_k(H(P\otimes s))(y)=H(Pds)(y)=P(Hy)ds(Hy)$$ 
for every $H\in\gl_N(\rr)$. In particular, the kernel $\ker^k_{pol}$ is $\gl_N(\rr)$-invariant. 

2) The mapping $g^-_k$ is equivariant under action of the orthogonal group $\oo(N)$: 
$$g^-_k(H(P\otimes s))=H((g^-_k)(P\otimes s))$$ 
for every $H\in\oo(N)$. 
In particular, the image $g^-_k(\ker^k_{pol})$ is $\oo(N)$-invariant.

3) The latter image is generated by derivatives of monomials $y^m$, $|m|=k-1$, along the 
vector fields $v_{ij}:=y_i\frac{\partial}{\partial y_j}-y_j\frac{\partial}{\partial y_i}$, $i\neq j$; each $v_{ij}$ 
generates the one-dimensional 
Lie algebra of the standard $\so(2)$-action on $\rr^N$ in the variables $(y_i,y_j)$.
\end{proposition}
\begin{proof} Statements 1) and 2) of the proposition follow by definition. The  image $g^-_k(\ker^k_{pol})$ 
is generated by the polynomials 
\begin{equation} g^-_k(Q_{m,i,j})=g^-_k(y^m(y_i\otimes s_{j}-y_{j}\otimes s_i))=y_i\frac{\partial y^m}{\partial y_j}-
y_j\frac{\partial y^m}{\partial y_i}, \ \ |m|=k-1,\label{gkqmi}\end{equation}
by Proposition \ref{ppbasis}. The latter right-hand side is the derivative of the polynomial $y^m$ along the 
generator $v_{ij}$ of the $\so(2)$-action on the variables $(y_i,y_j)$. This proves Proposition \ref{pequivar}.
\end{proof}
\begin{proposition} \label{derao} The derivatives from Proposition \ref{pequivar}, Statement 3), lie in $\mcp^{k-1}_0$.
\end{proposition}
\begin{proof} Averaging the derivative $\frac{dy^m}{dv_{ij}}$ along the $\so(2)$-action in the variables 
$(y_i,y_j)$ yields zero, analogously to the well-known fact that the  derivative of a function on a circle 
has zero average. Every function on the unit sphere  in $\rr^N$ 
having zero average along the above $\so(2)$-action 
has zero average on the whole unit sphere as well, since the volume form of the sphere (foliated by $\so(2)$-orbits) is the product of the 
family of length elements of the $\so(2)$-orbits and a measure transversal to the foliation by $\so(2)$-orbits. 
This proves Proposition \ref{derao}.
\end{proof}

\begin{proof} {\bf of Lemma \ref{ppeucl}.}  One has $g^-_k(\ker^k_{pol})\subset\mcp^{k-1}_0$ 
(Proposition \ref{derao} and Proposition \ref{pequivar}, Statement 3)). For even $k$ one has $\mcp^{k-1}_0=\mcp^{k-1}$, since 
each odd degree homogeneous polynomial has zero average along the unit sphere: the antipodal 
map changes the sign of such a polynomial. 
Therefore, it suffices 
to prove that each polynomial in $\mcp^{k-1}_0$ lies in $g^-_k(\ker^k_{pol})$. We prove this statement by 
induction in the number $N$ of variables.

Induction base: $N=2$. A homogeneous polynomial in $(y_1,y_2)$ has zero average along the unit circle, 
if and only if it is the derivative of another homogeneous polynomial by the vector field $v_{12}$ generating 
the $\so(2)$-action. Indeed, writing the restriction to the unit circle of 
a homogeneous polynomial of degree $k-1$ as a trigonometric 
polynomial in $\phi=\arctan(\frac{y_2}{y_1})$ of the same degree reduces the above statement to the 
following well-known one: a trigonometric polynomial of a given degree has zero average, if and only if 
it is the derivative of another trigonometric polynomial of the same degree. Therefore, the space $\mcp^{k-1}_0$ 
of polynomials in $(y_1,y_2)$ coincides with $g^-_k(\ker^k_{pol})$, by Proposition \ref{pequivar} (Statement 3))
and Proposition \ref{derao}.

Induction step. Let the statement $\mcp^{k-1}_0=g^-_k(\ker^k_{pol})$ be proved for all $N\leq d$, $N\geq2$. 
Let us prove it for $N=d+1$. Fix an arbitrary $P\in\mcp^{k-1}_0$. It can be represented as the sum 
$$P=Q+R, \ Q \text{ has zero average along the } \so(2) \text{ action in } (y_1,y_2);$$
$$ \text{the polynomial } \ \  R \ \ \text{ is } \so(2)-\text{invariant.}$$
Namely, $R$ is the average of the polynomial $P$ under the above $\so(2)$-action. Let us show that 
$Q$, $R$, and hence, $P$ lie in  $g^-_k(\ker^k_{pol})$. 
The polynomials $Q$ and $R$ are homogeneous of the same degree $k-1$. The polynomial 
$Q$ is a derivative, as above, see the proof of the induction base. Therefore,  $Q\in g^-_k(\ker^k_{pol})$, 
by Proposition \ref{pequivar}, Statement 3). One has 
$$R(y)=\sum_{j=0}^{[\frac N2]} (y_1^2+y_2^2)^jR_j(y_3,\dots,y_N),$$
$$ R_j \text{ are homogeneous 
polynomials, } \deg R_j=N-2j,$$
by $\so(2)$-invariance. 
For every $j\in\nn$ the $\so(2)$-average of the monomial $y_2^{2j}$ is equal to $c_j(y_1^2+y_2^2)^j$, $c_j>0$. 
Set 
$$\wt R(y):=\sum_{j=0}^{[\frac N2]} c_j^{-1}y_2^{2j}R_j(y_3,\dots,y_N).$$
The difference $\wt R-R$ has zero average along the $\so(2)$-action, by construction. Therefore, 
it lies in  $g^-_k(\ker^k_{pol})$, see the above argument. Now it remains to show that $\wt R\in g^-_k(\ker^k_{pol})$. 
The polynomial $\wt R$ has zero average along the  unit sphere in $\rr^N$, by assumption. 
On the other hand,  it is a polynomial of $N-1$ variables $y_2,\dots,y_N$. Therefore, it has zero average along 
the $(N-2)$-dimensional unit sphere in $\rr^{N-1}_{y_2,\dots,y_N}$. Hence, it lies in the $g^-_k$-image of 
the kernel $\ker^k_{pol}$ from the space $\mcp^k\otimes\lll$ in variables $(y_2,\dots,y_N)$, by the induction hypothesis. The induction step is over. 
Lemma \ref{ppeucl} is proved.
\end{proof}

Step 3: proof of Lemma \ref{pk=4} in the general case. Recall that $S^k(\ttg)$ is the space of sections 
of a smooth  vector bundle $E_k$ whose fiber over each point $x\in \gamma$ is the 
space of degree $k$ homogeneous polynomials on the tangent space $T_x\gamma$. 
The bundle $E_k$ is isomorphic to the $k$-th symmetric power of the cotangent bundle $T^*\gamma$, see the discussion 
before Remark \ref{rkrist}. To prove Lemma \ref{pk=4}, we show that for every even $k$ 
 each section of the bundle $E_{k-1}$ is a finite linear combination of the restrictions to 
 $T\gamma$ of elements of the image $G^-_k(\ker^k_{pol})$ with $C^{\infty}$-smooth coefficients. To this end, we will show 
 that the latter image spans the space of degree $k-1$ homogeneous polynomials at each point $x\in \gamma$: see the 
 following definition. 
\begin{definition} Let $\gamma$ be a $C^{\infty}$-smooth manifold (not necessarily compact). 
Let $\pi:E\to \gamma$ be a $C^\infty$-smooth finite-dimensional vector bundle. 
A (finite or infinite) collection of its sections $(f_i)_{i\in I}$ is called {\it generating} (for the bundle $E$), if for every $x\in \gamma$ the vectors
$f_i(x)\in E(x)$ span the fiber $E(x)$ over $x$.
\end{definition}
\begin{proposition} \label{progen} Let $k$ be an even number. 
The sections from the image $G^-_k(\ker^k_{pol})$ generate the vector bundle 
$E_{k-1}$. 
\end{proposition}
\begin{proof} Fix an arbitrary point $x\in \gamma$. We consider that the 
origin in $\rr^N$ is distinct from $x$ and the line connecting $x$ with 
the origin is transversal 
to $T_x\gamma$. One can achieve this by translation. Let us choose a  
Euclidean scalar product on $\rr^N$ and orthonormal  coordinates $(s_1,\dots,s_N)$ centered at 0  so that 

- the vector subspace parallel to $T_x\gamma$ be the coordinate subspace $\rr^n_{s_1,\dots,s_n}$; 

- the translation pushforward of the  Riemannian metric on $T_x\gamma$ to $\rr^n_{s_1,\dots,s_n}$ be the 
standard Euclidean metric given by $ds_1^2+\dots+ds_n^2$; 

- the radius vector of the point $x$ be orthogonal to $T_x\gamma$. 

One can achieve this by applying a linear transformation and choosing appropriate scalar product. 
These operations  change neither the space $\mcp^k\otimes\lll$, nor $\ker^k_{pol}$ (Proposition \ref{pequivar}, 
Statement 1)). 
Let us equip the tangent spaces to $\rr^N$ with the coordinates $y_1,\dots,y_N$ obtained from $s_1,\dots,s_N$ 
by translations. This identifies the tangent subspace $T_x\gamma$ with  $\rr^n_{y_1,\dots,y_n}$. 
For every $Q\in\mcp^k\otimes\lll$ containing only $s_j$  and $y_j$ with $j\leq n$ the image $G^-_k(Q|_{T\gamma})$
 coincides on the fiber $T_x\gamma$ with the restriction to it of the form $g^-_k(Q)$, by construction. 
 The $g^-_k$-image of the subspace in
$\ker^k_{pol}$ consisting of  polynomials in $y_1,\dots,y_n$ 
coincides with the similar subspace in $\mcp^{k-1}$,  
i.e., with the fiber at $x$ of the bundle $E_{k-1}$, 
by Lemma \ref{ppeucl}. This proves 
Proposition \ref{progen}. 
\end{proof}
\begin{proposition} \label{teogen} Let $\gamma$ and $\pi:E\to \gamma$, be as in the above definition. 
Let  $f_1,\dots,f_\ell$ be a finite generating collection of  sections of the bundle $E$. Then every $C^{\infty}$-smooth 
section of the bundle $E$ is a linear combination of the sections $f_1,\dots,f_\ell$ with coefficients 
being $C^{\infty}$-smooth functions on $\gamma$.
\end{proposition}
The author is sure that this proposition is well-known to specialists, but he did not 
find a  reference. 

\begin{proof} Set $d=\dim E$. 
Each point of the manifold $\gamma$ has a neighborhood $U$ such that there exists a collection of distinct indices 
$j_1,\dots,j_d$ for which the values $f_{j_1}(x),\dots,f_{j_d}(x)$ are linearly independent at every $x\in U$. 
Each section $F$ on such a neighborhood is a linear combination $F=\sum_{i=1}^d\eta_if_{j_i}$, where $\eta_i$ are 
$C^{\infty}$-smooth functions on $U$. If $F$ has compact support in $U$, then so do $\eta_i$. 

Now fix a locally finite at most countable covering of the manifold $\gamma$ by open subsets $U_\alpha$ as above 
and the corresponding partition of unity consisting of functions $\rho_\alpha$ compactly supported in $U_\alpha$. 
Let $F$ be an arbitrary $C^{\infty}$-smooth section of the bundle $E$. 
Each function $\rho_{\alpha}F$ is compactly supported in $U_\alpha$ and 
hence $\rho_\alpha F=\sum_{i=1}^d\eta_{i,\alpha}f_{j_{i,\alpha}}$, where $\eta_{i,\alpha}$  
are compactly supported in $U_\alpha$. 
Writing $F=\sum_\alpha(\rho_\alpha F)$ and replacing  $\rho_\alpha F$ by the latter linear combinations yields 
that $F$ is a linear combination of the sections $f_j$ with $C^{\infty}$-smooth coefficients.
The proposition is proved.
\end{proof}

\begin{proof} {\bf of Lemma \ref{pk=4}.} Each element  $h\in S^{k-1}(\ttg)$ is  a 
linear combination of the images $G^-_k(Q_{m,i,j}|_{T\gamma})$ with 
coefficients $\eta_{m,i,j}\in C^{\infty}(\gamma)$, by  
Propositions \ref{ppbasis}, \ref{progen} and  \ref{teogen}; $Q_{m,i,j}=y^my_i\otimes s_j-y^my_j\otimes s_i$, 
$|m|=k-1$. 
 The elements $h_{m,i,j}:=(\eta_{m,i,j}y^my_i)\otimes s_j-(\eta_{m,i,j}y^my_j)\otimes 
 s_i\in S^k(T^*\gamma)\otimes_\rr C^{\infty}(\gamma)$ 
 lie in $\ker G^+_k$, by construction. One has $G^-_kh_{m,i,j}=\eta_{m,i,j}G^-_k(Q_{m,i,j}|_{T\gamma})$, 
 since the operator $G^-_k$ is $C^{\infty}(\gamma)$-linear in the first tensor factor: $G^-_k(\eta\phi\otimes f)=
 \eta G^-_k(\phi\otimes f)$ for every $\eta\in C^{\infty}(\gamma)$, by definition, see (\ref{degk-}). 
 Therefore, $h=G^-_k(\sum_{m,i,j}h_{m,i,j})$. This proves Lemma \ref{pk=4} for the number $d_{k,n}$ being 
 equal to the number of elements $Q_{m,i,j}$ with $|m|=k-1$, $N=2n$: $d_{k,n}<(2N)^{k+1}$. 
\end{proof}
 \begin{proposition} \label{la40} The space $\La_3$ is contained in  $\{\La_4,\La_0\}$.
 \end{proposition}
 \begin{proof} The Poisson bracket space $\{\La_0,\La_4\}$ is contained in $\La_3\oplus\La_5$ (Proposition \ref{propoisf}). 
Fix an arbitrary $A=\frac{h(y)}{\sqrt{1+||y||^2}}\in\La_3$; $h\in S^3(\ttg)$. Let us show that 
$A\in\{\La_4,\La_0\}$. To do this, fix an element 
$$X:=\sum_{i=1}^d\phi_i\otimes f_i\ \in \ S^4(\ttg)\otimes_\rr C^{\infty}(\gamma), \ \ \ X\in\ker G^+_4, \ G^-_4(X)=h.$$
It exists by Lemma \ref{pk=4}. Set 
$$H_i:=\frac{\phi_i(y)}{\sqrt{1+||y||^2}}\in\La_4, \ F_i:=\frac{f_i(x)}{\sqrt{1+||y||^2}}\in\La_0, \ 
B:=\sum_{i=1}^d\{ H_i, F_i\}\in\La_3\oplus\La_5.$$
One has 
\begin{equation}\pi_3(B)=A, \ \ \pi_5(B)=\frac{(\nabla_y\phi)(y)}{\sqrt{1+||y||^2}}, \ \phi:=\sum_{i=1}^df_i\phi_i\in S^4(\ttg),\label{p3f3}\end{equation}
 by construction and (\ref{poishf}). Now replacing $B$ by 
 $$\wt B:=B-\{\frac{\phi(y)}{\sqrt{1+||y||^2}}, \frac1{\sqrt{1+||y||^2}}\}$$ \
 we cancel the remainder $\pi_5(B)$ in $\pi_5$ without changing $\pi_3$  and get that $\pi_5(\wt B)=0$, $\pi_3(\wt B)=
 \wt B=A\in\{\La_4,\La_0\}$. 
 The proposition is proved.
 \end{proof}
 
 \begin{proof} {\bf of Theorems \ref{tlidf1}.} Theorems \ref{tlidf1} follows from Corollary \ref{coroif} and 
 Proposition \ref{la40}.
 \end{proof}
 
 \subsection{Proofs of Theorems \ref{tlidf}, \ref{tlidense} and \ref{tlidense2} for $n\geq2$}
 \begin{proposition} \label{plode}  Let $U\subset\gamma$ be a local chart identified with a contractible domain in $\rr^n$. 
 The restrictions to $(T\gamma)|_U$ of functions from the space $\La:=\oplus_{k=0}^{\infty}\La_k$ are 
 $C^{\infty}$-dense in the space of $C^{\infty}$-functions on $(T\gamma)|_U$. 
  \end{proposition}
 \begin{proof} Let $z=(z_1,\dots,z_n)$ be the coordinates of the chart $U$. 
 The restricted tangent bundle $(T\gamma)|_U$ is the direct product $\rr^n_y\times U_z$. The restrictions to $(T\gamma)_U$ of functions from the space $\sqrt{1+||y||^2}\La$ are polynomials in $y$ with coefficients $C^{\infty}$-smoothly depending on $z$.   They include all the polynomials with coefficients being functions compactly 
 supported in $U$. The latter polynomials  are $C^{\infty}$-dense in the space of $C^{\infty}$-smooth functions on 
 $(T\gamma)|_U$. Indeed, the polynomials in $(y,z)$ are dense (Weierstrass Theorem). Multiplying them  
 by bump functions in $z$ compactly supported in $U$ we get polynomials of the above type 
 and conclude their density. 
   Therefore, $\sqrt{1+||y||^2}\La$ is dense there, and hence, so is $\La$. 
 \end{proof}

 \begin{proof} {\bf of Theorem \ref{tlidf}.} Fix a locally finite covering of the manifold $\gamma$ by 
 local contractible charts $U_\ell$. Let $(\rho_\ell)_{\ell=1,2,\dots}$ be a partition of unity, $\supp\rho_\ell\Subset U_\ell$. 
 Fix an arbitrary $h\in C^{\infty}(T\gamma)$. For every $\ell$ the function $h\rho_\ell$ has support projected 
 inside a compact subset  in $U_\ell$. Its restriction to $(T\gamma)_{U_\ell}$ 
  is a $C^{\infty}$-limit of the restrictions to $(T\gamma)|_{U_\ell}$ of functions  $h_{k,\ell}\in\La$, 
   by Proposition \ref{plode}. For every $\ell$ fix an arbitrary function 
 $\alpha_\ell(z)$ compactly supported in $U_\ell$ that is identically equal to 1  on a neighborhood 
 of $\pi(\supp(h\rho_\ell))$. 
 The functions $\wt h_{k,\ell}:=\alpha_\ell(z)h_{k,\ell}$ lie in $\La$ and converge to $h\rho_\ell$, as $k\to\infty$,
  in the $C^{\infty}$-topology 
 of the space of functions on $T\gamma$, by construction. 
 The sums $\sum_{\ell=1}^m\wt h_{k,\ell}$ lie in $\La$ for every $k$ and $m$
  and $C^{\infty}$-converge to $h$, as $k$ and $m$ tend to infinity, by construction and local finiteness of  covering. 
 Finally, $\La$ is  $C^{\infty}$-dense in $C^{\infty}(T\gamma)$. 
 
 Recall that the Lie  algebra $\gh$ consists of functions on $T_{<1}\gamma$. It is identified with $\Lambda$ by  
 the diffeomorphism $T_{<1}\gamma\to T\gamma$, $(x,w)\mapsto(x,y:=\frac{w}{\sqrt{1-||w||^2}})$, 
  by  Theorem \ref{tlidf1}. This together with density of the algebra 
 $\La$ implies density of the algebra $\gh$ in the space $C^{\infty}(T_{<1}\gamma)$ and proves Theorem \ref{tlidf}.
 \end{proof}

Theorems \ref{tlidense} and \ref{tlidense2}  follow immediately from Theorems \ref{tlidf} and \ref{forh}.

 \section{Density of pseudo-groups. Proofs of Theorems \ref{tc1}, \ref{tc2}, \ref{tds3} and Corollary \ref{c1}} 
 
 Let $\gamma$ be either a global strictly convex closed hypersurface in $\rr^{n+1}$, or a germ of hypersurface 
 in $\rr^{n+1}$.  
 It is supposed to be $C^{\infty}$-smooth. In what follows $\Pi$ will denote an open subset in the 
 space of oriented lines in $\rr^{n+1}$ where the reflection from the curve $\gamma$ 
 is well-defined: it is either the phase cylinder, as in Subsection 1.1, or diffeomorphic to a contractible domain 
 in $\rr^{2n}$, as in Subsection 1.2. First we prove Theorems \ref{tc1}, \ref{tc2} and Corollary \ref{c1}. Then we prove Theorem \ref{tds3}.
 
 \subsection{Density in Hamiltonian symplectomorphism pseudogroup. 
 Proofs of Theorems \ref{tc1}, \ref{tc2} and Corollary \ref{c1}} 
 \begin{definition} Let $M$ be a manifold, $V\subset M$ be an open subset, 
  $v$ be a smooth vector field on $M$, and $t\in\rr$, $t\neq0$. We say that the time $t$ flow map $g^t_v$ 
 is {\it well-defined on $V$,} if all the flow maps $g^{\tau}_v$ with $\tau\in(0,t]$ ($\tau\in[t,0)$, 
 if $t<0$) are well-defined 
 on $V$, that is, the corresponding differential equation has a well-defined solution for every initial 
 condition in $V$ for all the time values $\tau\in(0,t]$ (respectively, $\tau\in[t,0)$). 
 \end{definition}

 Recall that $\gh$ denotes the Lie algebra  of functions on $\Pi$ generated 
 by the space $\La_0$ of functions of type $H_f(x,w)=-2f(x)\sqrt{1-||w||^2}$ with respect to the 
 Poisson bracket for all $f\in C^{\infty}(\gamma)$. 
  
 \begin{proposition} \label{pmcg} Let $\mcf$ denote the pseudogroup generated by  flow maps 
 (well-defined on domains in the sense of the above definition) of the 
 Hamiltonian vector fields with Hamiltonian functions from the Lie algebra $\gh$. For every mapping 
 $\delta:C^{\infty}(\gamma)\to\rr_+$ the  $C^{\infty}$-closure $\overline{\mcg(\delta)}$ of the 
 corresponding pseudogroup $\mcg(\delta)$, see (\ref{pseudod}), contains $\mcf$.
 \end{proposition}
 \begin{proof} Let $L_0$ denote the space of Hamiltonian vector fields with Hamiltonian functions from the 
 space $\La_0$. The Lie algebra of Hamiltonian vector fields with Hamiltonians in $\gh$ consists of finite linear combinations of successive commutators of vector fields from the space $L_0$. 
 
 \medskip
 \noindent {\bf Claim 1.} {\it The well-defined 
 flow maps (on domains) of each iterated commutator  of a collection of vector fields in $L_0$ 
 are contained in $\overline{\mcg(\delta)}$.}
 
 \begin{proof} Induction in the number of Lie brackets in the iterated commutator. 
 
 Base of induction. The space $L_0$ consists of 
 the derivatives $v_h:=\frac{d}{d\var}\Delta\mct_{\var,h}$, $\Delta\mct_{\var,h}=\mct_{\gamma_{\var,h}}^{-1}
 \circ\mct_\gamma$, for all  $C^{\infty}$-smooth functions $h$ on $\gamma$. 
 Fix an arbitrary function $h\in C^{\infty}(\gamma)$, 
 a domain $W\subset\Pi$ and a $t\in\rr\setminus\{0\}$, say  
 $t>0$, such that the flow map $g^t_{v_h}$ is well-defined on $W$. 
Note that the mapping $F_\var:=(\Delta\mct_{\var,h})^{[\frac t\var]}$ converges to 
$g^t_{v_h}$, as $\var\to0$, whenever it is defined. Since $g^t_{v_h}$ is 
 well-defined on $W$, we therefore see that $F_\var$ is well-defined on 
 a domain $W_\var\subset\Pi$ such that $W_\var\cap W\to W$, as $\var\to0$, 
and $F_\var$ converges to $g^t_{v_h}$ on $W$ in the $C^{\infty}$ topology. Note that $F_\var\in\mcg(\delta)$ for $\var\leq\delta(h)$. 
 Hence, $g^t_{v_h}|_W\in\overline{\mcg(\delta)}$. Thus, the flow maps of the vector fields in $L_0$ 
 are contained in  $\overline{\mcg(\delta)}$. 
 
 Induction step. Let  $v$ and $w$ be vector fields on $\Pi$ such that their well-defined  
 flow maps lie in  $\overline{\mcg(\delta)}$. Let us prove the same statement for 
 their commutator $[v,w]$. Note that for small $\tau$ the flow maps 
 $g^\tau_v$, $g^{\tau}_w$, $g^t_{[v,w]}$ are well-defined on domains in $\Pi$ converging to $\Pi$, as 
 $\tau\to0$. Let  $t\in\rr\setminus\{0\}$, say $t>0$, and a domain $D\subset\Pi$ be such that the flow map $g^t_{[v,w]}$ 
 be well-defined on $D$. Then  $g^t_{[v,w]}|_D\in\overline{\mcg(\delta)}$, by a classical 
 argument: the composition 
 $(g^{\frac tN}_v\circ g^{\frac tN}_w\circ g^{-\frac tN}_v\circ g^{-\frac tN}_w)^{N^2}$ is well-defined 
 on a domain $W_N$ with $W_N\cap D\to D$, as $N\to\infty$; it belongs 
 to $\overline{\mcg(\delta)}$ and converges to $g^t_{[v,w]}$ on $D$  in the $C^{\infty}$-topology. 
 The induction step is done. The claim is proved.
 \end{proof}
  
 \noindent {\bf Claim 2.} {\it Let well-defined flow maps of vector fields $v$ and $w$ lie in $\overline{\mcg(\delta)}$. Then 
  well-defined flow maps of all their linear combinations also lie in $\overline{\mcg(\delta)}$.}
 
 \begin{proof} It suffices to prove the claim for the sum $v+w$, since $g^t_{cv}=g^{ct}_v$. The 
 composition  $(g^{\frac tN}_vg^{\frac tN}_w)^N$ obviously converges to $g^t_{v+w}$ on every domain where 
 the latter flow map is well-defined. 
 This proves the claim. 
 \end{proof}
 
 Claims 1 and 2 together imply  the statement of Proposition \ref{pmcg}.
 \end{proof}
 
 Recall the following well-known notion.
 \begin{definition} We say that a symplectomorphism $F$ of a symplectic manifold $M$ 
 has {\it  compact support} 
 if it is  identity outside some compact subset. 
 We say that it is {\it Hamiltonian with compact 
 support,} if it can be connected to the identity by a smooth path $F_t$ in the group of 
 symplectomorphisms $M\to M$, $F_0=Id$, $F_1=F$,
  such that the vector fields $\frac{dF_t}{dt}$ are Hamiltonian 
 with compact supports contained in one and the same compact subset.  
 \end{definition}

 \begin{proposition} \label{hamsup} The group of Hamiltonian symplectomorphisms with compact support is $C^\infty$-dense in the pseudogroup of $M$-Hamiltonian symplectomorphisms between domains 
 in the ambient manifold $M$.
 \end{proposition}
 
 \begin{proof} Consider  the Hamiltonian vector fields $\frac{dF_t}{dt}$ on domains $V_t$ 
 from the definition of a $M$-Hamiltonian 
 symplectomorphism $V\to U\subset M$ (Definition \ref{defmham}); let $g_t:V_t\to\rr$ denote the corresponding 
 Hamiltonian functions. We identify $V_t$ with $V$ by the maps $F_t$ 
 and consider $g_t$ as one function $g$ on  $V\times[0,1]_t$. We can approximate it 
 by functions $g_n$ with compact supports, $g_n\to g$ in the $C^{\infty}$-topology:    
 the convergence is uniform with all derivatives on compact subsets. For every $n$ the approximating function 
 yields a family of globally defined functions $g_{t,n}:M\to\rr$ with supports lying in the 
 same compact subset: the image of  $\supp g_n\Subset V\times[0,1]$ 
 under the map $(x,\tau)\mapsto F_{\tau}(x)$.  Let $v_{n,t}$ denote the Hamiltonian vector fields with Hamiltonians $g_{t,n}$. 
Then the time 1 map of the 
non-autonomous differential equation defined by $v_{t,n}$ converges to $F$ on $V$ 
in the $C^{\infty}$-topology, as $n\to\infty$. 
\end{proof}

\begin{proposition} \label{propham} The whole group of Hamiltonian symplectomorphisms $\Pi\to\Pi$ with compact 
support lies in the $C^{\infty}$-closure of the pseudogroup generated by well-defined flow maps of 
Hamiltonian vector fields with Hamiltonian functions from the Lie algebra $\gh$. 
\end{proposition}
\begin{proof} Consider a Hamiltonian symplectomorphism $F:\Pi\to\Pi$ 
with compact support, let $v(x,t)=\frac{dF_t}{dt}$ be the corresponding family of Hamiltonian 
vector fields with compact supports. The map $F$ is the $C^{\infty}$-limit 
of compositions of time $\frac1n$ 
flow maps of the autonomous Hamiltonian vector fields 
$$v(x,\frac kn), \ \ \ k=0,\dots,n-1,$$ 
as $n\to\infty$. Each of the above Hamiltonian vector 
 fields $v(x,\frac kn)$ can be approximated by Hamiltonian 
 vector fields with Hamiltonians in $\gh$, by Theorems \ref{tlidense} and \ref{tlidense2}. Then $F$ becomes 
 approximated by products of their flows. This implies the statement of Proposition \ref{propham}.
 \end{proof}

\begin{proof} {\bf of Theorem \ref{tc1}: global and local cases.} Each $\Pi$-Hamiltonian symplectomorphism between 
domains in $\Pi$ is a limit of a converging sequence of Hamiltonian symplectomorphisms $\Pi\to\Pi$ 
with compact 
supports in the $C^{\infty}$-topology, by Proposition \ref{hamsup}. The latter symplectomorphisms are, 
in their turn, limits of compositions of well-defined flow maps of vector fields with Hamiltonians from 
the algebra $\gh$, by Proposition \ref{propham}. For every mapping $\delta: C^{\infty}(\gamma)\to\rr_+$ 
the flow maps under question 
lie in $\overline{\mcg(\delta)}$, by Proposition \ref{pmcg}. This proves Theorem \ref{tc1} in global and local cases.
\end{proof}

\begin{proof} {\bf of Corollary \ref{c1} and Theorem \ref{tc2}: global and local cases.} 
Let us first consider the global case: $\gamma$ 
is compact.  Fix some $\alpha>0$ and $k\in\nn$. We can choose a 
mapping $\delta:C^{\infty}(\gamma)\to\rr_+$ so that for every $h\in C^{\infty}(\gamma)$ and every 
$\var\in[0,\delta(h))$ the hypersurface $\gamma_{\var,h}$, see (\ref{gaef}), is $(\alpha,k)$-close 
to $\gamma$. Then the statement of Corollary \ref{c1} follows immediately from Theorem \ref{tc1}. 
In the case, when $\gamma$ is local, we apply the above argument for functions $h$ with compact support. 
This proves Theorem \ref{tc2}.
\end{proof}

\subsection{Special case of germ of planar curve. Density in symplectomorphisms: proof of Theorem \ref{tds3}} 
To prove Theorem \ref{tds3}, which states similar approximability of (a priori non-Hamiltonian) 
symplectomorphisms in the case, when $\gamma$ is a germ of planar curve, we use the following three well-known propositions  
communicated to the author by Felix Schlenk. 

\begin{proposition} \label{exsm} Let $M$ be an oriented manifold. Let $V\subset M$ be 
an open subset with  smooth boundary, whose closure is compact and 
$C^{\infty}$-smoothly diffeomorphic to a closed ball $\overline B^n$. Then every $C^{\infty}$-smooth orientation-preserving diffeomorphism 
$F:\overline V\to F(\overline V)\Subset M$ extends to a $C^{\infty}$-smooth diffeomorphism $F:M\to M$ isotopic to 
identity with compact support. The isotopy can be chosen $C^{\infty}$-smooth and so that its diffeomorphisms coincide 
with the identity outside some (one and the same) compact subset.
\end{proposition} 
\begin{proof} (Felix Schlenk). {\bf Step 1.}
It suffices to show that

\begin{lemma}
The space of smooth orientation-preserving embeddings $\varphi \colon \overline B^n \to M$ is connected.
\end{lemma}

Indeed, given $V$ as in the proposition, 
choose a diffeomorphism $\varphi \colon \overline B^n \to V$.
By the lemma, we find a smooth family $\varphi_t \colon \overline B^n \to M$
of smooth embeddings, $t \in [0,1]$,
with $\varphi_0=\varphi$ and $\varphi_1 = F \circ \varphi$.
Consider the vector field $v(x,t) = \frac{d}{dt} \varphi_t(x)$
that generates this isotopy.
The vector field $v$ is defined on the compact subset 
$\left\{ (x,t) \mid x \in \varphi_t(\overline B^n) \right\}$ of $M \times [0,1]$.
Choose any smooth extension $\tilde v$ of $v$ to $M \times [0,1]$
that vanishes outside a compact set. 
Then the flow maps of $\tilde v$, $t \in [0,1]$, form the desired isotopy.

\medskip
\noindent 
{\bf Step 2.} Proof of the lemma.

\medskip \noindent
Let $\varphi_0, \varphi_1 \colon \overline B^n \to M$
be two smooth orientation-preserving embeddings of the closed unit ball 
$\overline B^n=\overline B^n_1$.
Take a smooth isotopy $g_t$, $t \in [0,1]$, of $M$
that moves $\varphi_0(0)$ to $\varphi_1(0)$.
Choose $\varepsilon >0$ such that 
$$
g_1 \bigl( \varphi_0 (\overline B_\varepsilon^n) \bigr) \,\subset\, \varphi_1(\overline B^n) .
$$
Consider the smooth family of embeddings $\overline B^n_1 \to M$ defined by
\begin{equation} \label{Alexander.partial}
\varphi_0^s (x) \,=\, \varphi_0(sx), \quad s \in [\varepsilon, 1] .
\end{equation}
By definition, $\varphi_1$ is the restriction of a smooth embedding 
$\tilde \varphi_1$ of the open ball $B_{1+\delta}^n$.
We now have two smooth orientation-preserving embeddings $g_1 \circ \varphi_0^\varepsilon, \varphi_1 \colon \overline B_1^n \to \tilde \varphi_1 (B_{1+\delta}^n)$.
Since $\tilde \varphi_1 (B_{1+\delta}^n)$ is diffeomorphic to $\RR^n$,
the classical Alexander trick shows that the embeddings
$g_1 \circ \varphi_0^\varepsilon, \varphi_1$
can be connected, see e.g.\ \cite[Appendix A, formula (A.1)]{schlenk}.
Since $g_1 \circ \varphi_0^\varepsilon$ is connected to $\varphi_0^\varepsilon$
by $g_t\circ\varphi_0^\var$, and since $\varphi_0^\varepsilon$ is connected to $\varphi_0$
by \eqref{Alexander.partial},
the lemma and  the proposition are proven. 
\end{proof}

\begin{proposition} \label{propschl} Let $M$ be a two-dimensional 
symplectic manifold, and let $V\subset M$ be 
an open subset with compact closure and smooth boundary. 
Let $\overline V$ be $C^{\infty}$-smoothly diffeomorphic to a 
disk. Then every $C^{\infty}$-smooth symplectomorphism $F:\overline V\to F(\overline V)\Subset M$ can be extended 
to a $C^{\infty}$-smooth symplectomorphism $F:M\to M$ with compact support. Moreover, $F$ can be chosen isotopic 
to the identity via $C^{\infty}$-symplectomorphisms with compact supports contained in one and the same 
compact subset in $M$. 
\end{proposition}  

\begin{proof} (Felix Schlenk). 
Consider an arbitrary extension $G$ of $F$ to a diffeomorphism $G:M\to M$ 
with compact support that is isotopic to the identity through diffeomorphisms $G_s:M\to M$   
 with compact supports contained in a common compact subset $K\subset M$, 
 $K\supset(\overline V\cup F(\overline V))$; $s\in[0,1]$, $G_0=Id$, $G_1=G$. It exists by 
Proposition \ref{exsm}. Set $W:=M\setminus K$. 
 Let $\omega$ be the symplectic form of the manifold $M$. Consider its pullback
 $\alpha:=G^*(\omega)$, which is a symplectic form on $M$. 
 The forms $\alpha_t:=t\alpha+(1-t)\omega$, $t\in[0,1]$, are also symplectic: they are 
 obviously closed, and they are positive area forms, since $M$ is two-dimensional. 
 This is the place we use two-dimensionality.  The forms $\alpha_t$ are  cohomologous to $\omega$, 
 the corresponding areas of the manifold $M$ are equal, and $\alpha_t\equiv\omega$ 
 on $\overline V\cup W$.  One has $\alpha_0=\omega$, $\alpha_1=\alpha$. 
 There exists a family of diffeomorphisms $S_t:M\to M$, $t\in[0,1]$, $S_0=Id$, 
 such that $S_t^*(\alpha_t)=\omega$ and $S_t$ coincide with the identity on the set $V\cup W$, 
  where $\omega=\alpha_t$. This follows from the relative version of Moser's deformation argument 
  \cite[exercise 3.18]{mcs}; see also \cite{moser},  \cite[p.11]{hof-zehn-mos}.  The composition 
 $\Phi:=G\circ S_1:M\to M$ preserves the symplectic form $\omega$ and coincides with $G$ on $V\cup W$. 
 Hence,  $\Phi|_V=F$, $\Phi|_W=Id$. 
 
 Applying the above construction of the symplectomorphism $\Phi$ to the diffeomorphisms $G_s$ instead of $G$
 yields a smooth isotopy $\Phi_s$ of the symplectomorphism $\Phi=\Phi_1$ to the identity via 
 symplectomorphisms with supports in $K$. (But now $G_s$ is not necessarily symplectic on $V$,
and   $(\Phi_s)|_V$ is not necessarily the identity.) The proposition is proved.
 \end{proof}
 
\begin{proposition} \label{symham} Every symplectomorphism  with compact support of an open 
topological disk equipped with a symplectic structure 
 is a Hamiltonian symplectomorphism with compact support. 
 \end{proposition}
 
 \begin{proof} Let $\omega$ denote the symplectic form. 
 Every symplectomorphism in question is isotopic to the identity via a 
 smooth family $F_t$ of symplectomorphisms with compact support, by  Proposition \ref{propschl}. 
 The derivatives $X_t:=\frac{dF_t}{dt}$  form a $t$-dependent family of symplectic vector fields. Hence, $L_{X_t}\omega=d(i_{X_t}\omega)=0$, by Cartan's formula. Thus, 
 $i_{X_t}\omega$ is a closed 1-form and hence, exact, $i_{X_t}\omega=dH_t$, since the underlying 
 manifold is a topological disk. Therefore, the vector fields  $\frac{dF_t}{dt}$ are Hamiltonian 
 with compact support. This proves the proposition.
 \end{proof}
 
 \begin{remark} Proposition \ref{symham} is a part of the following well-known fact:  
 {\it the group of compactly-supported symplectomorphisms of an open symplectic 
 topological disk is contractible and 
  path-connected by smooth paths.} The same statement holds for 
 symplectomorphisms of $\rr^4$ (Gromov's theorem \cite[p.345, theorem 9.5.2]{mcs2}). In higher dimensions 
 it is not known whether a similar statement is true. 
 For every $k\in\nn$ there exists an exotic symplectic structure on $\rr^{4k}$ admitting a symplectomorphism 
$\rr^{4k}\to\rr^{4k}$ with compact support  that is not smoothly isotopic to the identity in the class of 
(not necessarily symplectic) diffeomorphisms \cite[theorem 1.1]{cks}.
 \end{remark}

\begin{proof} {\bf of Theorem \ref{tds3}.} Let $\gamma$ be a germ of planar curve. Let $V\subset\Pi$ be an arbitrary simply connected domain. 
Let us show that 
each symplectomorphism $F:V\to F(V)\subset\Pi$ can be approximated by elements of the pseudogroup 
$\mcg(\delta)$ for every $\delta$. Consider an exhaustion $V_1\Subset V_2\Subset\dots$ 
of the domain $V=V_{\infty}$ by simply connected domains with compact closures in $V$ and smooth 
boundaries.  Each restriction $F|_{V_n}$ extends to a symplectomorphism 
$\Pi\to\Pi$ with compact support, by 
Proposition \ref{propschl}. The latter extension is Hamiltonian with compact support, by  Proposition \ref{symham}. 
The elements of the pseudogroup $\mcg(\delta)$ accumulate to the restrictions 
to $V$ of all the Hamiltonian symplectomorphisms $\Pi\to\Pi$, by Theorems \ref{tc1}, \ref{tc2}. Therefore, some sequence  in $\mcg(\delta)$ converges to $F$ 
on $V$ in the $C^{\infty}$-topology. This proves the analogue of Theorem \ref{tc1} with density
 in the pseudogroup of 
symplectomorphisms between simply connected domains in $\Pi$. 
The proof of similar analogue of Corollary \ref{c1} repeats the proof of Corollary \ref{c1} with obvious 
changes. Theorem \ref{tds3} is proved.  
  \end{proof}
 \section{Proof for hypersurfaces in Riemannian manifolds}
 Here we prove Theorems \ref{forh1} and \ref{theriem}.
 
 \begin{proof} {\bf of Theorem \ref{forh1}.} Fix a point $x\in\gamma$. Let $V(x)$ be its neighborhood in $M$ 
 equipped with normal coordinates centered at $x$ for the metric of the ambient manifold $M$. We deal with two metrics 
 on $V$: the metric of the manifold $M$ (which will be denoted by $g$) and the Euclidean metric in the 
 normal coordinates (which will be denoted by $g_{E}$). 
Consider the reflections $\mct_{\gamma_\var}$ with $\var\in[0,\delta]$ as mappings acting on compact subsets 
in the unit ball bundle $T_{<1}\gamma$  (after   identification (\ref{idgeo}), see Subsection 1.4). The mapping $\mct$ of reflection from $\gamma$ in the metric $g$ has the same 1-jet, as 
the reflection from $\gamma$ in the metric $g_E$, at each point of the fiber over $x$: at each point $(x,w)$, $w\in T_x\gamma$, with $||w||:=||w||_g=||w||_{g_E}<1$. This holds, since 
the  metrics in question have the same 1-jet at $x$. Similarly, 
the mappings $\mct_{\gamma_\var}$ constructed for both metrics have the same 1-jets at points $(x,w,0)$ as functions 
on $(T_{<1}\gamma)\times[0,\delta]$. Therefore, 
the derivatives $v_f=\frac{d\Delta\mct_{\var}}{d\var}|_{\var=0}$ calculated for both metrics coincide at all points 
$(x,w)$ with $w\in T_x\gamma$, $||w||<1$. 
This together with Theorem \ref{forh} applied to the Euclidean metric $g_E$ implies that  $v_f$ coincides with 
the  Hamiltonian 
vector field of the Hamiltonian function $H_f(x,w)=-2\sqrt{1-||w||^2}f(x)$ at each point $(x,w)$ as above. 
Hence, the field $v_f$ calculated for the given metric $g$ coincides with the latter Hamiltonian vector field 
everywhere, since the choice of the point $x$ was arbitrary. 
Theorem \ref{forh1} is proved.
\end{proof}
 
 \begin{proof} {\bf of Theorem \ref{theriem}.} The Lie algebra generated by the Hamiltonian functions $H_f(x,w)$ 
 of the above vector fields $v_f$ on $T_{<1}\gamma$ is dense in the space of all the $C^{\infty}$-smooth 
 functions on $T_{<1}\gamma$ (Theorem \ref{tlidf}). Therefore, the Lie algebra generated by the vector fields $v_f$ 
 is dense in the Lie algebra of Hamiltonian vector fields on $\Pi$ (in both cases, when $\gamma$ is either a closed 
 hypersurface, as above, or a germ). Afterwards the proof of Theorem \ref{theriem} repeats the arguments from 
 Section 4. 
 \end{proof}

    \section{Acknowledgements}
 
 I wish to thank Sergei Tabachnikov, Felix Schlenk and Marco Mazzucchelli for helpful discussions. I wish 
 to thank Felix Schlenk for  
 a  careful reading of the paper and very helpful remarks and suggestions.


\begin{thebibliography}{}


\bibitem{ar2} Arnold, V. {\it Mathematical methods of classical mechanics.}  
Springer-Verlag, 1978. 

\bibitem{ar3} Arnold, V. {\it Contact geometry and wave propagation.} 
Monogr. \textbf{34} de l'Enseign. Math., Universit\'e de G\'en\`eve, 1989. 

\bibitem{ban} Banyaga, A. {\it The structure of classical diffeomorphism groups.}  
Mathematics and its Applications, \textbf{400}. 
Kluwer Academic Publishers Group, Dordrecht; Boston, 1997.

\bibitem{cks} Casals, R.; Keating, A.; Smith, I.   
{\it Symplectomorphisms of exotic discs.}  
J. \'Ecole Polytechnique, \textbf{5} (2018), 289--316.

\bibitem{hof-zehn-mos} Hofer, H.; Zehnder, E. {\it Symplectic Invariants and Hamiltonian Dynamics.} 
Birkh\"auser, Basel, 1994.

\bibitem{mm} Marvizi S., Melrose R. {\it Spectral invariants of convex planar regions.} 
J. Diff. Geom. \textbf{17} (1982), 475--502.




\bibitem{mcs} McDuff, D.; Salamon, D. {\it Introduction to symplectic topology.} 
Clarendon Press, Oxford, Oxford University Press, New York,  1998. Second edition.

\bibitem{mcs2} McDuff, D.; Salamon, D. {\it J-holomorphic curves and symplectic topology.} 
AMS Colloquium Publications, \textbf{52} (2012). Second edition. 

\bibitem{melrose1} Melrose, R. {\it Equivalence of glancing hypersurfaces.} Invent. Math., 
\textbf{37} (1976), 165--192.

\bibitem{melrose2} Melrose, R. {\it Equivalence of glancing hypersurfaces 2.} 
Math. Ann. \textbf{255} (1981), 159--198. 

\bibitem{moser} Moser, J. {\it On the volume elements on a manifold.} Trans. Amer. Math.
Soc. \textbf{120} (1965) 286--294.

\bibitem{peirone} Peirone, R. {\it Billiards in tubular neighborhoods of manifolds of codimension 1.} 
Comm. Math. Phys. \textbf{207:1} (1999), 67--80.

\bibitem{perline} Perline, R. {\it Geometry of Thin Films.} J. Nonlin. Science \textbf{29} 
(2019), 621--642.

\bibitem{ptt} Plakhov, A.; Tabachnikov, S.; Treschev D. {\it Billiard transformations of parallel flows: A periscope theorem.} J. Geom. Phys., \textbf{115:5} (2017), 157--166.

\bibitem{schlenk} Schlenk, F. {\it Embedding problems in symplectic geometry.}   
De Gruyter Expositions in Mathematics, \textbf{40} (2008). 

\bibitem{tab95} Tabachnikov, S. {\it Billiards.}  Panor. Synth. \textbf{1} (1995), vi$+$142.

\bibitem{tab} Tabachnikov, S. {\it Geometry and Billiards,} Amer. Math. Soc. 2005. 
\end{thebibliography}
\end{document}